%% file: main.tex
\documentclass[11pt]{amsart}
\usepackage[ascii]{inputenc}
\usepackage{mathtools}
\usepackage{mathpazo}

\usepackage[dvipsnames]{xcolor}
\usepackage[bookmarksnumbered,linktocpage,hypertexnames=false,colorlinks=true,linkcolor=NavyBlue,urlcolor=NavyBlue,citecolor=ForestGreen,anchorcolor=green,breaklinks=true,pagebackref=true,pdfusetitle]{hyperref}
\usepackage[normalem]{ulem}
\usepackage[T1]{fontenc}
\usepackage{textcomp}
\usepackage[english]{babel}
\usepackage[capitalize,nameinlink]{cleveref}

\usepackage{graphicx}   
\usepackage{subcaption} 

\usepackage{amsmath,amsthm,amssymb,mathtools,setspace,dsfont,hyperref,color}
\usepackage{multirow,booktabs}
\usepackage[shortlabels]{enumitem}
\usepackage{faktor}
\usepackage{mathrsfs}
\usepackage{kantlipsum}
\usepackage{enumitem}
\usepackage{bm}
\usepackage{todonotes}
\usepackage{algorithm,algpseudocodex}
\usepackage{bm}

\newcommand{\crefpart}[2]{%
  \hyperref[#2]{\namecref{#1}~\labelcref*{#1}~\ref*{#2}}%
}

\renewcommand{\thispagestyle}[1]{}

\newtheorem{theorem}{Theorem}[section]
\newtheorem{corollary}[theorem]{Corollary}
\newtheorem{definition}[theorem]{Definition}

\newtheorem{lemma}[theorem]{Lemma}
\newtheorem{proposition}[theorem]{Proposition}

\newtheorem{remark}{Remark}
\newtheorem{example}{Example}

\newtheorem*{theorem*}{Theorem}
\newtheorem*{corollary*}{Corollary}
\newtheorem*{proposition*}{Proposition}

\newcommand{\eps}{\varepsilon}
\newcommand{\f}{\bm{f}}
\newcommand{\g}{\bm{g}}

\newcommand{\RR}{\mathbb{R}}
\newcommand{\CC}{\mathbb{C}}
\newcommand{\ZZ}{\mathbb{Z}}

\newcommand{\NN}{\mathbb{N}}
\newcommand{\dd}{\bm{d}}
\newcommand{\ff}{\bm{f}}
\newcommand{\HH}{\mathcal{H}}
\newcommand{\PP}{\mathcal{P}}

\DeclareMathOperator{\conv}{conv}

\DeclareMathOperator{\rk}{rk}
\DeclareMathOperator{\GL}{GL}
\DeclareMathOperator{\SL}{SL}
\DeclareMathOperator{\PD}{PD}

\DeclareMathOperator{\diag}{diag}
\DeclareMathOperator{\Diag}{Diag}
\DeclareMathOperator{\tr}{tr}
\DeclareMathOperator{\Lie}{Lie}
\DeclareMathOperator{\ii}{i}
\DeclareMathOperator{\proj}{proj}
\DeclareMathOperator{\Herm}{Herm}

\newcommand{\rmU}{\mathrm{U}}

\DeclarePairedDelimiter\norm{\lVert}{\rVert}

\newcommand{\rquo}[1]{(#1)_{\mathtt{R}}}

\keywords{Convex optimization; preconditioning; condition number; iterative methods}  

\begin{document}

\title{Optimal Preconditioning is a Geodesically Convex Optimization Problem \\ 
}

\author{M. Levent Do\u{g}an}
\thanks{Ruhr University Bochum, Bochum, Germany (mahmut.dogan@rub.de)} 
\author{Alperen Erg{\"u}r}
\thanks{The University of Texas at San Antonio, Texas, U.S.A. (alperen.ergur@utsa.edu)}
\author{Elias Tsigaridas}
\thanks{INRIA Paris and Sorbonne Universit\'{e},
Paris, France (elias.tsigaridas@inria.fr)}

\begin{abstract} 
	We introduce a unified framework for computing approximately-optimal preconditioners for solving linear and non-linear systems of equations.
	We  demonstrate that the condition number minimization problem, under structured transformations such as diagonal and block-diagonal preconditioners, is geodesically convex with respect to unitarily invariant norms, including the Frobenius and Bombieri--Weyl norms. 
	This allows us to introduce efficient first-order algorithms with precise convergence guarantees.
		
For linear systems, we analyze the action of symmetric Lie subgroups $G \subseteq \GL_m(\CC) \times \GL_n(\CC)$ on the input matrix and prove that the logarithm of the condition number is a smooth geodesically convex function on the associated Riemannian quotient manifold. 
We obtain explicit gradient formulas, show Lipschitz continuity, and prove convergence rates for computing the optimal Frobenius condition number: $\widetilde{O}(1/\eps^2)$ iterations for general two-sided preconditioners and $\widetilde{O}(\kappa_F^2 \log(1/\eps))$ for strongly convex cases such as left preconditioning. 

We extend our framework to consider preconditioning of polynomial systems 
	 $\f(x) = 0$, where $\f$ is a system of multivariate polynomials.
	 We analyze the local condition number $\mu(\f, \xi)$, at a root $\xi$ and prove that it also admits a geodesically convex formulation under appropriate group actions. We deduce explicit formulas for the Riemannian gradients and present convergence bounds for the corresponding optimization algorithms. To the best of our knowledge, this is the first preconditioning algorithm with theoretical guarantees for polynomial systems.

Our work relies on classical and recent results at the intersection of the theory of Lie groups with geodesically convex and non-commutative optimization. We explain the necessary building blocks for readers versed in numerical computing without assuming a background in algebra and geometry.

	Finally, we support our theoretical developments by numerical experiments. 
    In particular, for linear systems, we show that block-diagonal preconditioners significantly outperform the diagonal ones, and they consistently attain a better Frobenius condition number on large-scale random instances.
\end{abstract}  
\maketitle

\newpage

\tableofcontents

\newpage

\input{elaborate-intro}

\section{Preconditioning through the lens of geodesic convexity} \label{vocab}


\subsection{A primer on geodesically convex optimization}

Geodesic convexity generalizes the notion of convexity from Euclidean spaces to smooth Riemannian manifolds. 
The analogue of straight lines in Euclidean spaces are geodesics.
A subset $D \subset \mathcal{M}$ of a smooth Riemannian manifold $\mathcal{M}$ is called geodesically convex if every geodesic segment connecting two points in $D$ lies entirely within $D$.
	\footnote{On general Riemannian manifolds, geodesic convexity has competing definitions. However, we will only work on Hadamard manifolds---complete, simply connected spaces of nonpositive curvature---where any two points are connected by a unique geodesic. This eliminates ambiguity, as the geodesic convexity of a set, simply means that the unique geodesic between any two of its points remains entirely within it.}
A function $f: D \rightarrow \RR$ is geodesically convex if its domain is geodesically convex and, for any geodesic segment in $D$, the average value of $f$ at the endpoints is greater than or equal to its value at the midpoint. 
That is, $f$ becomes (Euclidean) convex by a restriction to a geodesic.
When $\mathcal{M}=\RR^n$ is the Euclidean space with its standard differential structure, the geodesically convex sets and functions coincide with usual convex sets and functions.
If the function $f:\mathcal{M}\rightarrow\RR$ is smooth and the manifold $\mathcal{M}$ is Riemannian, then the geodesic convexity of $f$ is equivalent to the positive semi-definiteness of its Hessian at every point in its domain.

Geodesically convex functions inherit desirable properties from their Euclidean counterparts:
\begin{itemize}
	\item Every local minimum is a global minimum.
	\item Gradient descent algorithms converge to the global optimum under appropriate conditions.
\end{itemize}

Geodesically convex optimization has attracted significant attention in recent years, with applications in diverse areas such as statistics \cite{Hosseini-Sra-15, Hosseini-Sra-20, Franks-Moitra-20, NSYKW-19}, operator and tensor scaling \cite{AZGLOW-18, BFGOWW-18, BFGOWW-19}, geometric programming \cite{BKVH-07, BLNW-20} and computing maximum entropy distributions \cite{Gurvits-06, Singh-Vishnoi-14, Straszak-Vishnoi-19}. 
For comprehensive studies of geodesically convex optimization on manifolds, we refer to \cite{Udriste-94, Sato-21, Boumal-23}, and for a survey of its applications in theoretical computer science, see \cite{Vishnoi-18}.

The results of B\"urgisser et al.~\cite{BFGOWW-19} 
are particularly relevant in our work,
especially because they express the 
problem of finding the minimal norm vector in an orbit of a group action as a geodesically convex optimization problem.
Specifically, under certain technical conditions, if a group $G$ acts linearly on a vector space $V$ and $v\in V\setminus\{0\}$, then the optimization problem 
\begin{equation}
\label{eq:norm-minimization}
\inf_{g\in G} \log\Vert g\cdot v\Vert
\end{equation} 
is geodesically convex. 
Here, $\Vert\cdot\Vert$ is a norm induced from a unitarily invariant inner product on $V$. This led to first- and second-order algorithms to compute the optimum value \cite{BFGOWW-19}.

We temporarily postpone the detailed definition of a geodesically convex function on a group $G$, 
to exploit the similarity of the optimization problem in \eqref{eq:norm-minimization} with preconditioning. 

\medskip

Following the discussion in \Cref{sec:elaborate-intro}, 
the optimal preconditioning for a full rank matrix $A$ 
consists in minimizing the condition number 
$\kappa(XAY^{-1})$ over a group $(X,Y)\in G$.
The group $G$ could be, for example, the group of invertible diagonal or block diagonal matrices.

To exploit a similarity with \eqref{eq:norm-minimization}, first, we 
consider the problem of
minimizing the logarithm of the condition number, $\log\kappa(A)$.
Then, we consider two actions of the group $G := \GL_m(\CC)\times\GL_n(\CC)$ of potential preconditioners on the space of matrices.
That is, 
(i) $G$ acts  on $\CC^{m\times n}$ via \emph{left-right multiplication}, i.e., $(X,Y)\cdot A \coloneqq XAY^{-1}$, and 
(ii) $G$ acts on $\CC^{n\times m}$ via \emph{right-left multiplication}, i.e., $(X,Y)\cdot B \coloneqq Y B X^{-1}$.
For fixed matrices $A\in\CC^{m\times n}$ and $B\in\CC^{n\times m}$,
the following is the simultaneous norm minimization problem \begin{equation}
\label{eq:simultaneous-norm-min}
	\inf_{(X,Y)\in G} \log \Vert XAY^{-1}\Vert_1 + \log \Vert Y B X^{-1}\Vert_2,
\end{equation}
where $\Vert\cdot\Vert_1$, respectively $\Vert\cdot\Vert_2$, is an arbitrary matrix norm on $\CC^{m\times n}$, respectively $\CC^{n\times m}$. 
Notice that $\log\Vert A\Vert_1 + \log\Vert B\Vert_2=\log\kappa(A,B)$ where $\kappa(A,B)\coloneqq\Vert A\Vert_1 \Vert B \Vert_2$ is the \emph{cross condition number} of $A$ and $B$, with respect to $\Vert\cdot\Vert_1$ and $\Vert\cdot\Vert_2$.
If $B=A^\dag$, then we recover the condition number of $A$, as 
$\kappa(A) = \kappa(A,A^\dag)$.
Hence, with $B=A^\dag$, the optimal value of \cref{eq:simultaneous-norm-min} is precisely $\inf_{(X,Y)\in G} \log\kappa(XAY^{-1})$.
Thus, \eqref{eq:simultaneous-norm-min}
reformulates the optimal condition number
as a simultaneous norm minimization problem
and so as the optimal value of a geodesically optimization problem. 

In the special case that $\Vert\cdot\Vert_i$, $i\in\{1,2\}$ is the \emph{Frobenius norm}, defined as$\Vert A\Vert_F\coloneqq\tr(A^{\ast}A)$, this reformulation allows us to use the results and algorithms from \cite{BFGOWW-19} to compute the optimal value.
This is due to the following observation:
We can view the Frobenius condition number as $\kappa_F(A)\coloneqq\Vert A\Vert_F \Vert A^{\dag}\Vert_F=\Vert A\otimes A^\dag\Vert$, where the norm on $A\otimes A^\dag$ is induced by the tensor product structure on $\CC^{m\times n}\otimes\CC^{n\times m}$, where each tensorand is endowed with the Frobenius norm. 
In this case, $\log\kappa_F(XAY^{-1})$ fits perfectly into the framework of norm minimization \cref{eq:norm-minimization}, since this norm on $\CC^{m\times n}\otimes\CC^{n\times m}$ is induced by a unitarily invariant inner product.
In fact, our results generalize this framework in the sense that we allow condition numbers that arise from arbitrary unitarily invariant matrix norms $\Vert\cdot\Vert_1$ and $\Vert\cdot\Vert_2$, such as the Euclidean condition number (\cref{thm:main-g-convexity} below).
Namely, we prove that the function $(X,Y)\mapsto \kappa(XAY^{-1})$ is a geodesically convex function, whenever the norms $\Vert\cdot\Vert_1$ and $\Vert\cdot\Vert_2$ are unitarily invariant matrix norms.

Similarly for polynomial systems, the minimization of the local condition number  $\mu(\f,\xi)$ is also captured by geodesically convex optimization: Consider the action of $\GL_m(\CC)\times \GL_n(\CC)$ on the set $\CC[x_1,\dots,x_n]^{\oplus m}$ of systems of $m$-polynomials as in \eqref{eq:non-lin-X-act} and \eqref{eq:non-lin-Y-act}.
For fixed $\f\in\CC[x_1,\dots,x_n]^{\oplus m}$ and $A\in\CC^{m\times n}$ the following
\[
\inf_{(X,Y)\in G} \, \log\Vert (X,Y)\cdot \f\Vert_W\,  +\,  \log\Vert Y A X^{-1}\Vert_2
\]
is a simultaneous norm minimization problem, and we prove that it is geodesically convex if $\Vert\cdot\Vert_2$ is a unitarily invariant matrix norm.
Here, $\Vert\cdot\Vert_W$ is the Bombieri-Weyl norm (\cref{def:BW-norm-system}).
If we set $A=D_{\xi}^{\dag}(\f)$, then 
we recapture the local condition number $\mu(\f,\xi)$.
As in the linear case, if $\Vert\cdot\Vert_2$ is the Frobenius norm, then we can rewrite the local condition number can be reformulated as $\mu(\f,\xi)\coloneqq \Vert \f\otimes D^\dag_{\xi}(\f)\Vert$, with the norm induced by a unitarily invariant inner product; this allows us to use results of \cite{BFGOWW-19} to optimize it.

\subsection{Precise statement of geodesic convexity results}
\label{sec:tech-details}

We now delve into the technical details that provide us the context to state our main results. 

We assume that $(X,Y)$ lives in a symmetric Lie subgroup $G \subseteq \GL_m(\CC)\times\GL_n(\CC)$. 
Thus, $G$ is a Zariski-closed subset of $\GL_m(\CC) \times \GL_n(\CC)$,\footnote{That is, $G$ is cut out by rational functions in the entries of the pairs of matrices in $\GL_m(\CC)\times\GL_n(\CC)$.} and for any $(X, Y) \in G$, also $(X^\ast, Y^\ast) \in G$,
where $X^\ast$ and $Y^\ast$ are the conjugate transposes.

The space of preconditioners we consider is quite expressive; it includes:
\begin{itemize}
    \item Groups $G = G_1 \times G_2$, where $G_1$ and $G_2$ are groups of invertible diagonal, block-diagonal, orthogonal, or symplectic matrices.
    \item Groups $G_1\times G_2$, where $G_1$ is one of the previous groups, and $G_2$ is the trivial group, or vice versa.
    This corresponds to preconditioning via multiplication either only from the left or only from the right.
    \item Groups $G = \{(g, g^{-1}) \mid g \in G_1\}$,
      when $m = n$ and   $G_1$ is one of the groups of the first case. 
    This corresponds to the conjugation action of $G_1$ on the space of matrices.
    \item Groups $G = \{(g, g^T) \mid g \in G_1\}$,
      when $m=n$ and $G_1$ is one of the groups in the first case.
    This corresponds to preconditioning $A$ with a similarity transformation, i.e., finding $g\in G_1$ with minimal $\kappa(gAg^T)$.
    \end{itemize}

Regarding the matrix norm, we assume that it is a unitarily invariant matrix norm on $\CC^{m \times n}$, such as the Schatten $p$-norm, which is 
\begin{equation}
	\label{eq:schatten-norm}	
\Vert A \Vert^S_p \coloneqq \Big( \sum\nolimits_{i=1}^{\min(m,n)} \sigma_i(A)^p \Big)^{\frac{1}{p}}, \qquad 1 \leq p \leq \infty
\enspace,
\end{equation}
where $\sigma_1(A) \geq \sigma_2(A) \geq \dots \geq \sigma_{\min(m,n)}(A)$ are the singular values of $A$. 
For $p = 2$, we obtain the Frobenius norm $\Vert A \Vert^S_2 = \Vert A \Vert_F = \sqrt{\tr(A^\ast A)}$, and for $p = \infty$, we obtain the operator norm $\Vert A \Vert^S_\infty = \Vert A\Vert = \sigma_1(A)$.
More generally, every unitarily invariant matrix norm is a \emph{symmetric gauge function} of the singular values of the matrix.

Assume that $\Vert\cdot\Vert_1$, respectively $\Vert\cdot\Vert_2$, is a unitarily invariant matrix norm on $\CC^{m\times n}$, respectively $\CC^{n\times m}$.
The corresponding condition number is \[
\kappa(A) = \Vert A\Vert_1 \, \Vert A^\dag \Vert_2.
\]
We aim to minimize $\kappa(A)$ under the action of $G$. 
The matrix norms are unitarily invariant, so $\kappa(UAV)=\kappa(A)$, for unitary matrices $U$ and $V$.
Therefore, the action of the \emph{maximal compact subgroup} 
\[ K\coloneqq G\cap \big(\mathrm{U}_m(\CC)\times \mathrm{U}_n(\CC)\big) \] 
leaves $\kappa(A)$ invariant and we can ignore the action of $K$.
This allows us to assume without loss of generality that we optimize $\kappa(A)$ over the set 
\[
\rquo{G / K} \coloneqq \{Kg \mid g\in G\}
\]
of \emph{right cosets} of $K$. 
For $g=(X,Y)\in G$, a right coset of $K$ is 
$Kg \coloneqq \{(UX, VY) \mid (U,V)\in K\}$.

Our focus is on the properties of the function:
\begin{equation}
\label{eq:C_A-def}
\begin{split}
    C_A : \rquo{G/K} &\rightarrow \RR \\
    Kg &\mapsto \log\kappa(g\cdot A) = \log \Vert XAY^{-1}\Vert + \log \Vert YA^{\dag}X^{-1}\Vert,
\end{split}
\end{equation} 
where $g=(X,Y)\in G$, and by $g\cdot A$ we denote the action of $g$ on $A$, i.e., $g\cdot A\coloneqq XAY^{-1}$.
The domain of $C_A$ is not a Euclidean space and $C_A$ is not convex. 
However, it turns out that $C_A$ is convex when it is restricted to certain curves in $\rquo{G/K}$, which correspond to \emph{geodesics}. 
To be more precise, we can show that $\rquo{G/K}$ admits a smooth Riemannian manifold structure, and when restricted to geodesics of this manifold $C_A$ is convex.

Indeed, $\rquo{G/K}$ is a symmetric space with non-positive sectional curvature, see \cite[Chapter~10]{Petersen-16}.
The tangent space of $\rquo{G/K}$ at the coset $K$ is naturally identified with $\Lie(G)/\Lie(K)$, where we define the Lie algebra as \[
\Lie(G) \coloneqq \{ (H_1, H_2) \in\CC^{m\times m} \oplus \CC^{n\times n} \; \mid \; (e^{tH_1} , e^{tH_2}) \in G \text{ for all }t\in\RR \},
\] where $e^{tH_1},e^{tH_2}$ denote the matrix exponential. 
The Lie algebra of $K$ is defined analogously and it equals the skew-Hermitian part of $\Lie(G)$ \[
\Lie(K) = \{(H_1,H_2)\in\Lie(G)\mid H_1^\ast+H_1 =0,\;  H_2^\ast+H_2=0\}.
\]
Then, $\Lie(G)$ admits a decomposition as \[
\Lie(G) = \Lie(K) \oplus \ii \Lie(K).
\] This is an orthogonal decomposition with respect to the Frobenius inner product on $\CC^{m\times m}\oplus\CC^{n\times n}$, namely \begin{equation}
\label{eq:Frobenius-inner}
\langle (H_1,H_2),(H_3,H_4)\rangle \coloneqq \mathrm{Re}\left(\tr(H_1^\ast H_2) + \tr(H_3^\ast H_4)\right).
\end{equation}
Hence, $\ii\Lie(K)$ is the Hermitian part of $\Lie(G)$, i.e., \[
\ii\Lie(K) = \Lie(G)\cap \Herm(m)\oplus\Herm(n).
\] With respect to the inner product \cref{eq:Frobenius-inner} on the tangent space $\ii\Lie(K)=\Lie(G)/\Lie(K)$, $\rquo{G/K}$ becomes a Riemannian manifold.

The geodesics of $\rquo{G/K}$ are of the form
 \[
\gamma(t) \coloneqq K \, e^{tH}g, \qquad g=(X,Y)\in G,\; H=(H_1,H_2)\in \ii \Lie(K),
\]
where $\gamma$ connects $\gamma(0)=(X,Y)$ to 
$\gamma(1) = (e^{H_1} X , e^{H_2} Y )$.
Here, $e^H \coloneqq (e^{H_1},e^{H_2})$ is the short-hand notation for matrix exponential, applied to pairs of matrices.

\begin{definition}
\label{def:geodesic-convexity}
	A function $f: \rquo{G/K} \rightarrow \RR$ is \emph{geodesically convex} if $f\left(K e^{tH}g \right)$
	 is convex in the real parameter $t$, for every $g\in G$ and $H\in\ii\Lie(K)$.
\end{definition}

The main structural result of this paper is the following.

\begin{theorem}
\label{thm:main-g-convexity}
	Suppose $m,n\geq 1$ and $G\subseteq \GL_m(\CC)\times\GL_n(\CC)$ is a symmetric Lie subgroup with maximal compact subgroup $K\coloneqq G\cap (\rmU_m(\CC)\times \rmU_n(\CC))$.
    Furthermore, assume that $\Vert\cdot\Vert_1$ is a unitarily invariant matrix norm on $\CC^{m\times n}$ and $\Vert\cdot\Vert_2$ is a unitarily invariant matrix norm on $\CC^{n\times m}$.
    \begin{enumerate}[(i)]
        \item \label{thm:main-part-1}
    Consider a matrix $A\in\CC^{m\times n}$ and its condition number $
    \kappa(A)\coloneqq\Vert A\Vert_1 \Vert A^{\dag}\Vert_2$.
    Then, 
    \[
    C_A: \rquo{G/K} \rightarrow\RR,\quad K(X,Y)\mapsto \log\kappa(XAY^{-1})
    \] 
    is geodesically convex.
    \item \label{thm:main-part-2}
    For a polynomial system $\f\in\CC[x_1,\dots,x_n]^{\oplus m}$ and  $\xi\in\CC^n$, let $\mu(\f,\xi)=\Vert f\Vert_W \, \Vert D^{\dag}_{\xi}(\f)\Vert_2$ be the local condition number.
    Then,  
    \[
    C_{(\f,\xi)}: \rquo{G/K} \rightarrow\RR,\quad K(X,Y)\mapsto \log\mu((X,Y)\cdot \f, Y\xi)
    \] 
    is geodesically convex.
    \end{enumerate}
\end{theorem}
We prove the first part in \cref{sec:setup}, and the second part in \cref{sec:p-systems-and-condition}.


\subsection{First order algorithm for the optimal preconditioner}
\label{sec:first-order-linear}
The geodesic convexity of $C_A$ allows us to minimize it using algorithms based on first order methods, provided that $C_A$ is smooth.
Unfortunately, the operator norm $\Vert A \Vert$ and the Euclidean condition number $\kappa(A)$ are not smooth functions of $A$.
Instead, we will use the Frobenius norm $\Vert A\Vert_F \coloneqq \tr(A^\ast A)$ and the Frobenius condition number $\kappa_F(A)\coloneqq \Vert A\Vert_F \Vert A^\dag\Vert_F$.
We note that the Frobenius condition number can act as an approximation to the Euclidean condition number, due to the inequality $\kappa(A)\leq\kappa_F(A)\leq n\kappa(A)$.
Furthermore, our experiments reveal that optimizing the Frobenius condition number generally decreases the Euclidean condition number as well, see \cref{sec:detailed-experiments}, in particular \cref{fig:correlation-Gauss,fig:correlation-Sparse}.

\begin{definition}
    Let $f: \rquo{G/K} \to \RR $ be a smooth function.
    The \emph{Riemannian gradient} $\nabla f(K(X,Y))$ of $f$ at $K(X,Y)\in \rquo{G/K}$ is
     \[
    \langle \nabla f(K(X,Y)), (H_1,H_2) \rangle = \left. \frac{\mathrm{d}}{\mathrm{d}t} \right|_{t=0} \; f(K(e^{tH_1}X, e^{tH_2}Y)), \quad \text{for all }(H_1,H_2)\in\ii\Lie(K).
    \] 
\end{definition}
Here, the inner product on $\ii\Lie(K)$ is the Frobenius inner product \cref{eq:Frobenius-inner}.
The general gradient descent algorithm on $\rquo{G/K}$ takes the form \begin{equation}
\label{eq:gradient-descent}
	(X_{k+1}, Y_{k+1}) \, = \, \big( e^{ -\eta_k H_k } X_k , e^{ -\eta_k H'_k } Y_k \big)
\end{equation}
 where $k$ is the current iteration, $(H_k,H'_k) = \nabla C_A(X_k,Y_k)$ is the gradient of the objective function at the $k$-th element $(X_k,Y_k)$, and $\eta_k\in\RR_{>0}$ is the step-size. 
 It is natural to assume that the starting point $(X_0, Y_0) = (I_m, I_n)$ is the identity element.

There is a simple formula for $\nabla C_A(X,Y)$.
Let's denote $B\coloneqq XAY^{-1}$. 
Then, $C_A(X,Y) = \log\kappa(XAY^{-1}) = C_{B}(I_m,I_n)$. 
Hence, $\nabla C_A(X,Y) = \nabla C_{B}(I_m, I_n)$ by the definition of the gradient. 
The latter gradient lies in the Lie algebra of $G$ and it is uniquely defined by the formula \[
\langle \nabla C_B(I_m, I_n) , (H_1, H_2) \rangle = \frac{\mathrm{d}}{\mathrm{d} t}\Big|_{t=0} C_B ( e^{tH_1} , e^{tH_2} ), \qquad \text{for every }(H_1,H_2)\in\Lie(G),
\] where $\langle \cdot, \cdot\rangle$ is the Frobenius inner product on $\Lie(G)$, defined as in \cref{eq:Frobenius-inner}.
\begin{lemma}
\label{lem:gradient}
Let $A\in\CC^{m\times n}$ and $(X,Y)\in G$. Then, $C_A$ is a smooth function and its gradient at $(X,Y)$ is given by
 \begin{equation}
 \label{eq:gradient-formula}
 	\nabla C_A( X , Y ) = \nabla C_{B}(I_m, I_n) = \proj_{\Lie(G)}\Big( \,\frac{BB^\ast}{\Vert B\Vert_F^2} - \frac{(B^\dag)^\ast B^\dag}{\Vert B^\dag\Vert_F^2}\; , \; -\frac{B^\ast B}{\Vert B\Vert_F^2} + \frac{ B^\dag (B^\dag)^\ast}{\Vert B^\dag\Vert_F^2} \, \Big),
 \end{equation}
 where $B\coloneqq XAY^{-1}$ and $\proj_{\Lie(G)}:\CC^{m\times m}\oplus\CC^{n\times n}\rightarrow\Lie(G)$ denotes the orthogonal projection onto the Lie algebra of $G$.
 \end{lemma}
In certain cases, the orthogonal projection has a simple description. 
For example, when the group of preconditioners is $G=T_m(\CC)\times T_n(\CC)$ the group of diagonal matrices, its Lie algebra $\Lie(G)=\mathrm{Diag}_m(\CC)\times\mathrm{Diag}_n(\CC)$ consists of pairs of diagonal matrices, and the projection $\proj_{\Lie(G)}(H_1,H_2) =\big(\Diag(H_1) , \Diag(H_2) \big) $ simply maps $(H_1,H_2)$ to their diagonals. 
 
\Cref{lem:gradient} gives a characterization of the optimal preconditioner $(X,Y)\in G$ in terms of $A$. 
Since $C_A$ is geodesically convex, its global minima, local minima and critical points all coincide. 
Hence, we have \[
(X,Y)\in G \text{ maximally reduces }\kappa_F(XAY^{-1}) \quad \iff\quad \nabla C_A(X,Y) = 0.
\]

Let's compute the gradient in a concrete setting:
Assume for simplicity that $A\in\CC^{n\times n}$ is an invertible matrix, and set $G=T_m(\CC)$, acting on $A$ by left multiplication.
Hence, we consider diagonal left preconditioners. 

In this case, the Lie algebra of $G$ consists of $m$ by $m$ (not necessarily invertible) diagonal matrices, i.e., $\Lie(G)=\Diag_m(\CC)$, where $\Diag_m(\CC)$ denotes the $m\times m$ diagonal matrices.
The orthogonal projection $\proj_{\Lie(G)}(M)=\Diag(M)$ simply maps $M$ to its diagonal, i.e. $\Diag(M)$ is obtained by setting the off-diagonal entries of $M$ to zero.

Let's assume that $X\in T_m(\CC)$ is positive definite.\footnote{The preconditioners can always be assumed to be positive definite by the polar decomposition $G=KP$, where $P$ is the set of positive definite matrices in $G$.}
Then, \cref{lem:gradient} implies that the gradient of $C_A$ at $X$ equals \[
\nabla C_A(X) = \Diag\left(\, \frac{XAA^\ast X}{\Vert XA\Vert_F^2} - \frac{ X^{-1} (A^{-1})^\ast A^{-1} X^{-1} }{\Vert A^{-1} X^{-1}\Vert_F^2} \,\right).
\] In particular, the optimal diagonal preconditioner $X$ is the diagonal matrix that satisfies the equality $\Diag(XAA^*X^{\ast})=c \Diag(X^{-1} A^{-*}A^{-1}X^{-1})$ for some constant $c>0$. 

We note that $C_A$ is not a strongly convex function (\cref{def:mu-strongly-convex}) in general: 
If $\mathrm{Stab}_G(A)\coloneqq\{(X,Y)\in G\mid XAY^{-1}=A\}$ denotes the stabilizer of $A$, then it is easy to see that $\kappa_F(XAY^{-1})$ is constant on a right coset of $\mathrm{Stab}_G(A)$. 
In particular, $C_A$ is not strongly convex on $\rquo{G/K}$, unless $\mathrm{Stab}_G(A)\subset K$. 
In the case of left-preconditioners, acting on a rank $m$ matrix $A\in\CC^{m\times n}$, we indeed have $\mathrm{Stab}_G(A)=I_m$, and $C_A$ is then a strongly convex function.
This allows us to improve the dependence of the iteration complexity on $\varepsilon$ from $\varepsilon^{-2}$ to $\log(\varepsilon^{-1})$.
\begin{theorem}[Left preconditioners]
\label{thm:main-linear-algorithm}
   Let $G$ be a symmetric subgroup of $\GL_m(\CC)$. 
    Let $A\in\CC^{m\times n}$ be a full rank matrix,
    $\kappa^{\star}\coloneqq \inf_{X\in G} \kappa_F(XA)$
    be the optimal condition number we can achieve by preconditioning  $A$	with an element $X \in G$,
    and $\eps > 0$.
	There is a first-order (gradient descent) algorithm, with constant step-size $\eta=1/8$
	and the initial point $g_0\coloneqq (I_m,I_n)$, 
	that reaches a group element $X$, 
	such that $\log\kappa_F(X A) - \log\kappa_F^{\star}<\varepsilon$,
	in 
    \[
    T = O\Big(\;  \log\Big( \frac{\log\left(\kappa_F(A)/\kappa_F^\star\right)}{\varepsilon}\Big)\, \kappa^2_F(A) \;\Big)
    \quad \text{ iterations.}
    \] 
\end{theorem}
The proof of \cref{thm:main-linear-algorithm} appears in \cref{sec:strongly-convex-linear}. Now we present the general result including left and right preconditioning.

\begin{theorem}
\label{thm:general-linear-complexity}
Let $G$ be either 
	(i) a symmetric subgroup of $\GL_m(\CC)\times\GL_n(\CC)$ with $T_m(\CC)\times T_n(\CC)\leq G$,
       where $T_m(\CC)$ is the group of diagonal elements,
       or (ii) a symmetric subgroup of $\GL_m(\CC)$ with $T_m(\CC)\leq G$. 
     
    Let $A\in\CC^{m\times n}$ be a full rank matrix,
    $\kappa_F^\star\coloneqq \inf_{g\in G}\kappa_F(g\cdot A)$
    be the optimal condition number we can achieve by preconditioning  $A$	with an element $g \in G$,
    and $\frac{1}{2}> \eps > 0$.
     
	There is a first-order (gradient descent) algorithm, with constant step-size $\eta=1/8$
	and the initial point $g_0\coloneqq (I_m,I_n)$, 
	that reaches a group element $X$, 
	such that $\log\kappa_F(g\cdot A) - \log\kappa_F^{\star}<\varepsilon$, \[
	T = O\Big(\;  \frac{\log\left(\kappa_F(A)/\kappa_F^\star\right)}{\varepsilon^2}\, \max(m^3,n^3) \;\Big) 
		\quad \text{ iterations.}
	\]
\end{theorem}
We will prove \cref{thm:general-linear-complexity} in \cref{sec:general-complexity-linear}. 

\cref{thm:general-linear-complexity} gives the iteration complexity of the gradient descent algorithm.
Below, we discuss the computational cost of each iteration, which consists of the following substeps:
1) Computing $B=X AY^{-1}$ where $(X,Y)\in G$,
2) evaluating the gradient of $C_A$ at $(X,Y)\in G$, using the formula \cref{eq:grad-def},
which involves inverting $BB^{\ast}$ and projecting it onto the Lie algebra of $G$,
3) and updating $(X,Y)\leftarrow e^{-\eta \nabla C_A(X,Y)} (X,Y)$ for a fixed step-size $\eta$.
For general groups, these steps can be computationally expensive, especially for very large $m,n$.
However, if $G$ has structure---e.g., the group of diagonal, or block-diagonal matrices with fixed block size---the operations simplify significantly.
In such cases, step 1 reduces to efficient matrix scaling, and step 3 involves a tractable matrix exponentiation (e.g., scalar exponentiation for diagonal, and matrix exponentiation of a matrix with fixed size for block diagonal preconditioning), where fast algorithms are known \cite{Sushant-Vishnoi-14}.

A particular difficulty lies in computing the gradient: 
By \cref{eq:grad-def}, this requires inverting the matrix $BB^{\ast}$, and projecting it onto the Lie algebra of $G$. Fortunately, for most groups of interest, e.g., block diagonal matrices, the Lie Algebra is small and the projection can be approximated without computing the full inversion. We explain the details of this approximation in \Cref{sec:detailed-approx-grad}.

\vspace{-0.05 in}
\subsection{Details on approximating the gradient}
\label{sec:detailed-approx-grad}

A critical component of the gradient descent algorithm is the computation of the gradient of the function $C_A$.
At the identity element, the gradient is given by
 \eqref{eq:gradient-formula} with $B = A$, because $X = I_m$ and $Y = I_n$ .

The initialization of the gradient descent consists of the projection of the matrices $(AA^\ast)^{-1}$ and $(A^\ast A)^{-1}$ onto the Lie algebra of $G$.
To explicitly compute and then project these inverses is computationally prohibitive for large-scale problems, as dense matrix inversion scales as $O(n^3)$ for an $n \times n$ matrix.
Instead, we employ efficient numerical algorithms for projecting without computing 
the whole inverse matrix and we exploit the sparsity of $A$.

When $\Lie(G)$ is a coordinate subspace, such as the case of diagonal/block-diagonal preconditioning, the problem of estimating the entries in a sparsity pattern of the inverse of a matrix using only matrix-vector products is a well-studied problem, see \cite{Saad-Tang-12,BKS-07,Dharan-Musco-23,ACKHMM-24} and references therein.
The main idea is to use Lanczos quadrature and its stochastic variant, Hutchinson's estimator, which we will now explain.

Let  $G$ be a subgroup of the group of pairs of invertible block-diagonal matrices, 
so $\Lie(G)$ is a subalgebra of the set of pairs of block-diagonal matrices.
Now, computing the gradient reduces to estimating the block-diagonal entries of $(AA^\ast)^{-1}$.

For deterministic approximation of individual blocks, the block Lanczos algorithm \cite{Golub-Underwood-77} is highly effective, e.g.~\cite{Chen-24}. 
Given a positive definite matrix $M$ and a collection $B=\begin{pmatrix}v_1,v_2,\dots,v_r\end{pmatrix}$ of vectors, the Lanczos algorithm constructs an orthonormal basis for the block Krylov subspace $\mathcal{K}_k(M, B)=\mathcal{K}_k(M, v_1) + \dots + \mathcal{K}_k(M, v_r)$ and and iteratively approximates the quadratic form $B^\ast M^{-1} B$-which equals e.g. the top left block of $M^{-1}$ when $v_i=e_i$ for $i=1,2,\dots,r$. 
The key advantage is that $k$-iterations of block Lancsoz algorithm requires only $kr$ matrix-vector products with $M$ to approximate $B^\ast M^{-1} B$.

We now show that $k=O(\sqrt{\kappa(M)}\log(\varepsilon^{-1}))$ iterations suffice for an $\varepsilon$-approximation: 
Convergence analysis shows that running block Lancsoz algorithm for $k$ iterations yields an approximation to $B^\ast M^{-1}B$ with additive error bounded by $\Vert B\Vert_F^2 \min_{\deg p < 2k} \sup_{x\in I} | x^{-1} - p(x) |$ where $I=[\lambda_{\min}(M),\lambda_{\max}(M)]$ is the interval containing all eigenvalues of $M$, see \cite[Corollary~9.11]{Chen-24}. 
It is well-known that the best polynomial approximation to $x^{-1}$ is given by Chebyshev polynomials of the first kind \cite{Meinardus-67} and taking $k=O(\sqrt{\kappa(M)}\log(\varepsilon^{-1}))$ yields an $\eps$-approximation to $x^{-1}$ on the interval $[\lambda_{\min}(M),\lambda_{\max}(M)]$ \cite{KVZ-12}. 
Thus, we can estimate a single block on the diagonal of $M^{-1}$ up to precision $\varepsilon$ with $O\left(\sqrt{\kappa(M)}\log(\frac{1}{\varepsilon})\right)$ iterations of the block Lancsoz algorithm.
Since each iteration costs $r$ matrix-vector products with $M$, we obtain the following:

\begin{theorem}
	\label{thm:inv-pd-diag}
    Let $M\in\CC^{n\times n}$ be a positive definite matrix. 
    Then, there is an algorithm which estimates a single $r\times r$-block on the diagonal of $M^{-1}$ within $\eps$-error in \[
    O\Big( r\, \sqrt{\kappa(M)}\, \log(\varepsilon^{-1}) \Big)
    \quad \text{ matrix-vector products with }  M.
    \] 
\end{theorem}

For a matrix $A$, the above theorem with $M=AA^\ast$ shows that we can approximate a single block on the diagonal of $M^{-1}$ with $O(r\, \kappa(A)\,\log(\eps^{-1}))$ matrix vector products with $A$ and $A^{\ast}$.
For well-conditioned sparse matrices, this is better than inverting $AA^\ast$.

If we require to compute all the diagonal elements, Hutchinson's stochastic estimator is a more scalable alternative \cite{BKS-07, Saad-Tang-12}. 
For a matrix $M$, Hutchinson's estimator approximates the diagonal $\Diag(M)$ as 
\begin{equation}
\label{eq:Hutchinson-estimator}
    \Diag(M) \approx \frac{1}{m}\; \sum_{i=1}^m z_i \,\odot\, M \, z_i \enspace,  
\end{equation}
where $z_i\in\{-1,1\}^n$ are i.i.d. Rademacher random vectors and $\odot$ denotes the Hadamard product. 

For $M=(AA^\ast)^{-1}$, each summand in \cref{eq:Hutchinson-estimator} requires the solution of the linear system $(AA^T) x = z_i$ for $i \in [m]$. 
Using the conjugate gradient method, we can solve each one in $O\left(\sqrt{\kappa(AA^\ast)}\log(\varepsilon^{-1})\right)=O\left(\kappa(A)\log(\varepsilon^{-1})\right)$ matrix-vector products with $A$ and $A^\ast$.
In particular, the Hutchinson's estimator \cref{eq:Hutchinson-estimator} can be computed in $O(m \kappa(A) \log(\frac{1}{\varepsilon}))$ matrix-vector products with $A$ and $A^\ast$. 
Thus, the full estimator requires $O\left(m \kappa(A)\log(\frac{1}{\varepsilon})\right)$ products, which is better than than $O(n^3)$ 
for well-conditioned sparse matrices.

Recent work \cite{Baston-Nakatsukasa-22, HIS-23,Dharan-Musco-23} analyses the estimator's error.
The variance of Hutchinson's estimator scales as $O(1/m)$ and $m=O(\varepsilon^{-2})$ is sufficient for a relative $\varepsilon$-error approximation with high probability.
Specifically, \cite{Dharan-Musco-23} shows that with constant probability, Hutchinson's estimator gives a vector $\tilde{d}$ that satisfies 
\[
\Vert \tilde{d} - \Diag(M)\Vert \leq c\sqrt{\frac{1}{m}} \, \Vert \overline{M}\Vert_F \enspace, 
\] 
where $\overline{M}=M-\Diag(M)$ denotes the matrix $M$ with diagonal entries are set to zero.

There is a generalization of Hutchinson's estimator for block diagonal matrices.
Amsel et al~\cite{ACKHMM-24} give an algorithm to estimate an $r\times r$-block of a matrix $M$ on the diagonal.\footnote{Their algorithm is in fact a stochastic estimation of the projection of the matrix onto any sparsity pattern.}
First, we consider a Gaussian matrix $G\in\RR^{n\times m}$, i.e., a matrix with entries sampled iid from the normal distribution.
Then, by setting $Z\coloneqq MG$, we can approximate the top-left $r\times r$-block of $M$ by 
\begin{equation}
    \label{eq:BlockHutchinson}
\text{TopLeftBlock}_{r\times r}(M) \approx G_r^{\dag} Z_r
\end{equation} 
where $G_r^T$ is the matrix consisting of the first $r$ rows of $G$ and $Z_r^T$ is the matrix consisting of the first $r$ rows of $Z$. 
When $r=1$, this estimate is equivalent to the Hutchinson's estimator with Gaussian random vectors.
With constant probability, \cref{eq:BlockHutchinson} gives an $r\times r$ matrix $B$ with \[
\Vert \text{TopLeftBlock}_{r\times r}(M) - B\Vert_F \leq O\left(\frac{1}{\sqrt{m}}\right) \, \Vert \overline{M}\Vert_F ,
\] where $\overline{M}$ is the matrix $M$ with top-left $r\times r$ block entries are set to zero. 
Note that computing the estimate \cref{eq:BlockHutchinson} requires only $m$ matrix-vector products with $M$. 

\begin{remark}
Using Hutchinson estimator for computing the gradient is essentially a stochastic gradient descent variant. Since our manifold $G / K$ has negative and tame curvature one can consider using any other stochastic gradient descent variant as well \cite{sra2016first}. 
\end{remark}

\subsection{Preconditioning non-linear systems}
\label{sec:intro-preco-non-linear}

In this section we extend our preconditioning framework to polynomial system of equations. 
While the condition number $\mu(\ff,\xi) \coloneqq \Vert \ff\Vert_W \Vert D^\dag_{\xi}(\ff)\Vert$ is geodesically convex for any unitarily invariant norm $\Vert\cdot\Vert$, it is typically non-smooth. 
We therefore focus specifically on the \emph{Frobenius local condition number}:
\[
\mu_F(\ff, \xi) \coloneqq \Vert \ff\Vert_W \Vert D^{\dag}_{\xi}(\ff)\Vert_F,
\]
which maintains both convexity and smoothness properties essential for our optimization approach.

\subsubsection*{Shuffling Preconditioning for Polynomial Systems}

We first consider preconditioning through the \emph{shuffling action} of $\GL_m(\CC)$ (as defined in \eqref{eq:non-lin-X-act}). 
The key insight is that this problem can be reduced to matrix preconditioning through an appropriate reformulation.

Let us recall that $\Vert \ff\Vert_W$ denotes the Bombieri-Weyl norm of $\ff$, given by:
\[
\Vert \ff\Vert_W = \sqrt{\sum_{i=1}^m \langle f_i, f_i \rangle},
\]
where $\langle \cdot, \cdot \rangle$ is the Bombieri-Weyl inner product (Definition~\ref{def:BW-inner}). Let $G_{\ff}$ be the Gram matrix of $\ff$ with respect to this inner product, i.e., $(G_{\ff})_{ij} = \langle f_i, f_j \rangle$. 

The connection to linear system preconditioning becomes clear through the following observations:
\begin{itemize}
    \item The action of $X \in \GL_m(\CC)$ transforms the Gram matrix as $G_{X\cdot\ff} = X G_{\ff} X^\ast$
    \item For $S_{\ff} \coloneqq \sqrt{G_{\ff}}$, we have $\Vert \ff\Vert_W^2 = \tr(G_{\ff}) = \Vert S_{\ff}\Vert_F^2$
    \item The condition number can be expressed as $\mu_F(\ff,\xi) = \Vert S_{\ff} \Vert_{F} \Vert D^{\dag}_{\xi}(\ff)\Vert_F$
    \item The shuffling action satisfies $S_{X\cdot\ff} = X S_{\ff}$
\end{itemize}

This formulation yields the elegant expression:
\[
\mu_F(X\cdot\ff, \xi) = \Vert X S_{\ff}\Vert_F \Vert D^{\dag}_{\xi}(\ff) X^{-1}\Vert_F,
\]
allowing us to directly apply our results for linear system preconditioning.


\begin{theorem}
\label{thm:shuffling-iteration}
Suppose $G\leq\GL_m(\CC)$ is a symmetric subgroup with $T_m(\CC)\leq G$. 
Let $\f=(f_1,\dots,f_m)\in\CC[x_1,\dots,x_n]^m$ be a polynomial system and $\xi\in\CC^n$.
Set $\mu_F^\star\coloneqq \inf_{X\in G}\mu_F(X\cdot f,\xi)$.
Let $\varepsilon>0$.
Then, in \[
T=O\,\Bigg(\; \frac{\log(\mu_F(\f,\xi)/\mu_F^\star)}{\varepsilon^2}\, \max(m^3,n^3)\;\Bigg)
\] steps, the first order iteration with step-size $\eta=\frac{1}{8}$ and $X_0=I_m$, reaches a group element $X$ with \[
\log\left(\frac{\mu_F(X\cdot\ff,\xi)}{\mu_F^\star}\right)<\varepsilon.
\]
\end{theorem}
We prove \cref{thm:shuffling-iteration} in \cref{sec:shuffling-algo}.

\subsubsection*{Shuffling and Linear Change of Variables: General Preconditioning}

We now present our most general result, considering both shuffling and variable transformations through the optimization of:
\[
C_{(\ff,\xi)}(X,Y) \coloneqq \log\mu_F\left( (X,Y)\cdot\ff, Y\xi \right),
\]
over a symmetric subgroup $G \leq \GL_m(\CC) \times \GL_n(\CC)$.

\begin{theorem}
\label{thm:change-iteration}
Suppose $G\leq\GL_m(\CC)\times\GL_n(\CC)$ is a symmetric subgroup with $T_m(\CC)\times T_n(\CC)\leq G$. 
Let $\ff=(f_1,\dots,f_m)\in\CC[x_1,\dots,x_n]^m$ be a polynomial system with \[
D\coloneqq \max_{i} \deg(f_i)
\] and $\xi\in\CC^n$.
Set $\mu_F^\star\coloneqq \inf_{(X,Y)\in G}\mu_F((X,Y)\cdot f,Y\xi)$.
Let $\frac{1}{2}>\varepsilon>0$.
Then, in \[
T=O\,\Bigg(\; \frac{\log(\mu_F(\f,\xi)/\mu_F^\star)}{\varepsilon^2}\,(m+n)^2\, \big(D+2\big)^{2m+2n-1}\;\Bigg)
\] steps, the first order iteration with step-size $\eta=(D+2)^{-1}$ and $(X_0,Y_0)=(I_m,I_n)$, reaches a group element $(X,Y)$ with \[
\log\left(\frac{\mu_F((X,Y)\cdot\ff,Y\xi)}{\mu_F^\star}\right)<\varepsilon.
\]
\end{theorem}
We prove \cref{thm:change-iteration} in \cref{sec:general-poly-algo}. One needs to note that the change of variables action changes all coefficients, and therefore acts on a space with dimension roughly $m \binom{d+n}{d}$. Even if one consider a linear system with $m \binom{n+d}{d} > m \; d^n$ (assuming $n>d$) many entries, preconditioning will reflect this dimensionality in the obtained complexity estimates. Therefore, we consider $D^{2m+2n-1}$ natural, however, it could be potentially improved to its $4$th root.
\subsection*{Torus Action for Sparse Polynomials}
We briefly mention that another approach to preserve sparsity is to consider a torus action on the variables in addition to shuffling equations. Our framework also handles this action as we show it is geodesically convex in  \Cref{sparse}.

\section{Experimental evaluation}
\label{sec:additional-experiments}

We present experimental results to 
evaluate the performance of our preconditioning algorithms.
Our implementation
relies  on the Julia packages \texttt{Manopt.jl} and \texttt{Manifolds.jl} \cite{Manoptjl,Manifoldsjl}, which enable efficient geodesically convex optimization over matrix manifolds. Our (quite simple) code is available at \href{https://github.com/mldogan/Preconditioning-via-g-convex-optimization}{our git page}. 

We evaluate preconditioners on $425$ sparse matrices from the \texttt{SuiteSparse} collection \cite{Kolodziej2019}, each with $\leq 1200$ rows. 
For each matrix in the collection, we compute both the optimal diagonal left-preconditioner and the optimal block-diagonal left-preconditioner, using a fixed block size of $5 \times 5$.
The results appear in \cref{fig:Sparse-Frobenius-BlockvsDiag-Copy},
and demonstrate consistent improvement of the condition number for both schemes:
\begin{itemize}
\item The Frobenius condition number $\kappa_F$ improves by $\approx 1602\times$ on average with block-diagonal preconditioning, compared to $\approx 601\times$ for diagonal preconditioning.
\item The block-diagonal preconditioner's improvement is $2.43\times$ greater than that of the diagonal preconditioner.
\end{itemize}

\begin{figure}[htbp]
  \centering
  \begin{subfigure}[b]{0.32\textwidth}
    \centering
    \includegraphics[width=\textwidth]{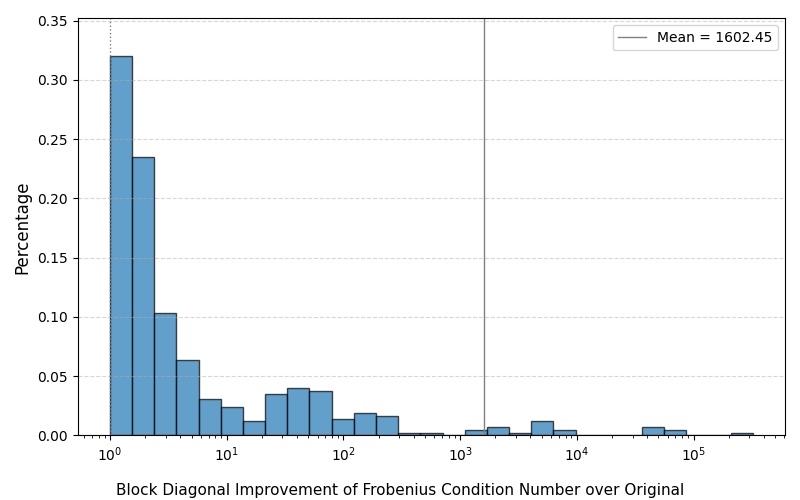}
    \phantomsubcaption
  \end{subfigure}
  \hfill
  \begin{subfigure}[b]{0.32\textwidth}
    \centering
    \includegraphics[width=\textwidth]{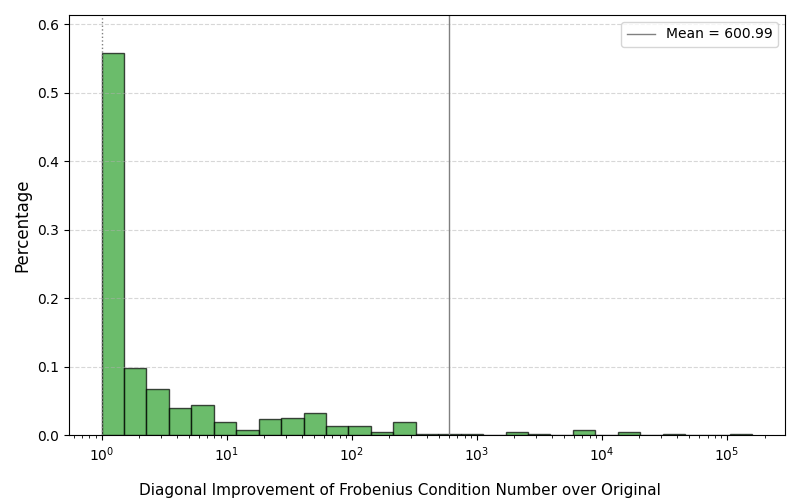}
    \phantomsubcaption
  \end{subfigure}
  \hfill
  \begin{subfigure}[b]{0.32\textwidth}
    \centering
    \includegraphics[width=\textwidth]{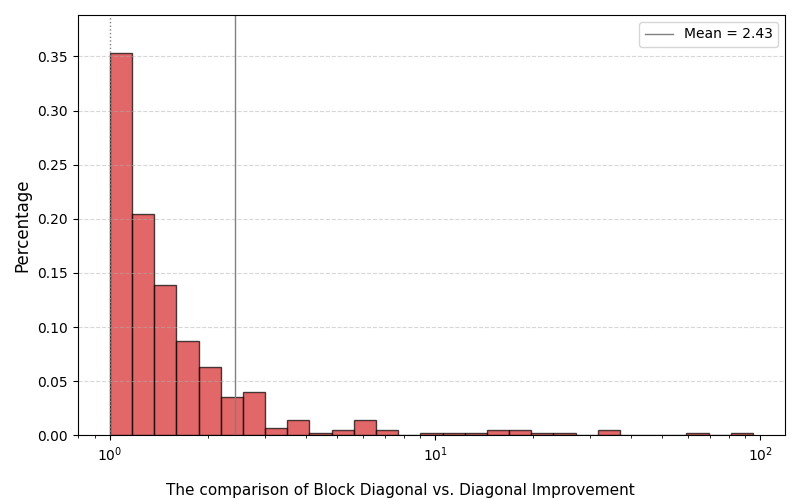}
    \phantomsubcaption
  \end{subfigure}
  \caption{The figure shows the improvement (defined as $\kappa_F(A)/\kappa_F^\star$) of the Frobenius condition number for the block diagonal preconditioner with $5\times 5$ blocks (blue) and for the diagonal preconditioner (green) for a collection of $425$ \texttt{SuiteSparse} matrices. On average, the block diagonal preconditioner with $5\times 5$-blocks improve the Frobenius condition number by $\approx 1602$ whereas the diagonal preconditioner improves it by $\approx 601$. The red plot shows the ratio of the improvement of the block-diagonal preconditioner vs. diagonal preconditioner. On average, the improvement of the optimal block-diagonal preconditioner is $\approx 2.43$ times the improvement of the optimal diagonal preconditioner.}
  \label{fig:Sparse-Frobenius-BlockvsDiag-Copy}
\end{figure}

Some additional experimental results, along with further performance metrics and figures, appear in \cref{sec:detailed-experiments}.

\newpage

\section{Geodesic Convexity in Matrix Groups} \label{hardcore}
In this section we collect facts and lemmas we will use later in the proofs.

\subsection{Groups and representations}
We start by introducing the geometric and algebraic foundations for geodesically convex optimization on matrix groups. 
We focus on symmetric subgroups of $\GL_m(\CC)$ and their representations, which provide the setting for our optimization framework. 
For a comprehensive treatment of symmetric Lie groups and their geometry, we refer to \cite{Wallach-17}, while our notation and key results on geodesic convexity follow \cite{BFGOWW-19}.

A subgroup $G \leq \GL_m(\CC)$ is called \emph{symmetric} if it is Zariski-closed (i.e., defined by polynomial equations in the matrix entries $g_{ij}$ and $\det^{-1}(g)$) and closed under conjugate transposition ($g \in G \implies g^* \in G$). 
The Lie algebra of $G$, denoted $\Lie(G)$, consists of all matrices $H \in \CC^{m \times m}$ such that $e^{tH} \in G$, for all $t \in \RR$. 
The Lie algebra is naturally identified with the tangent space of $G$ at the identity.

 We denote by $\rmU_m(\CC)=\{g\in \GL_m(\CC)\mid gg^\ast = I_m\}$ the unitary group. 
For a symmetric group $G$, the intersection $K \coloneqq G \cap \rmU_m(\CC)$ is a \emph{maximal compact subgroup} of $G$, its Lie algebra consists of skew-Hermitian matrices in $\Lie(G)$:
\[
\Lie(K) = \{H\in\Lie(G)\mid H^\ast + H=0\}.
\] 
The Lie algebra of $G$ decomposes as a real vector space 
 \[
\Lie(G) = \Lie(K) \oplus \ii\Lie(K),
\]
 where $\ii \Lie(K)$ consists of Hermitian matrices. 
This decomposition plays a crucial role in the geometry of $G$
and it is orthogonal with respect to the Frobenius inner product
 \begin{equation}
\label{eq:Frob-inner}
\langle H_1, H_2 \rangle = \mathrm{Re}(\tr(H_1^* H_2)).
\end{equation}
Note that \eqref{eq:Frob-inner} is a positive definite inner product on $\ii\Lie(K)$.

Let $V$ be an $N$-dimensional complex vector space.
A \emph{rational representation} of $G$ is a group homomorphism $\pi \colon G \to \GL(V)$, where the matrix entries (with respect to any basis of $V$) are rational functions of $g_{ij}$ and $\det^{-1}(g)$. 
The associated Lie algebra representation $\Pi \colon \Lie(G) \to \mathrm{End}(V)$ is \begin{equation}
\label{eq:Lie-alg-action}
\Pi(H) v = \left. \frac{\mathrm{d}}{\mathrm{d}t} \right|_{t=0} \pi(e^{tH})v ,
\end{equation} 
and satisfies $e^{\Pi(H)} = \pi(e^H)$ for all $H \in \Lie(G)$.

Assume that $G$ contains the diagonal torus $T_m(\CC) \cong (\CC^\times)^m$, the subgroup of $\GL_m(\CC)$ that consists of invertible diagonal matrices.
When restricted to the action of $T_m(\CC)$, $V$ decomposes into weight spaces:
\[
V = \bigoplus_{i=1}^N V_{\omega_i} \enspace,
\]
where $\omega_i \in \ZZ^m$ are the \emph{weights} of $\pi$, and $V_{\omega_i}$ is the one-dimensional simultaneous eigenspace for the action of $T_m(\CC)$ with character $t \mapsto t^{\omega_i} = t_1^{\omega_{i1}} \cdots t_m^{\omega_{im}}$. 
That is, $t\cdot v = t^{\omega_i} v$ for every $v\in V_{\omega_i}$.
The \emph{weight matrix} $M \in \ZZ^{N \times m}$ encodes the weights as rows.

\subsection{Geodesically convex functions}
Let's consider \[
\rquo{G/K} \coloneqq \{Kg\mid g\in G\},
\] the set of right cosets of $K$.
The Lie group structure of $G$ makes it a smooth manifold.
The transitive right action of $G$ on $\rquo{G/K}$ defines the structure of a smooth manifold on $\rquo{G/K}$, with the natural map $G\rightarrow\rquo{G/K}$ being a smooth submersion.
The tangent space of $\rquo{G/K}$ at the identity (i.e., the coset $K$) is naturally identified with $\Lie(G)/\Lie(K)\cong \ii\Lie(K)$,
and the Frobenius inner product \eqref{eq:Frob-inner} makes $\rquo{G/K}$ a Riemannian manifold.
It is known that $\rquo{G/K} = {Kg \mid g \in G}$, equipped with this Riemannian metric, is a symmetric space with non-positive sectional curvature \cite[Chapter~10]{Petersen-16}. 

The general form of a geodesic of $\rquo{G/K}$ is \[
\gamma : \RR\rightarrow \rquo{G/K}, \quad \gamma(t)=K e^{tH}g 
\] for $H\in\ii\Lie(K)$. 
On $\rquo{G/K}$, we study functions that are convex along geodesics:
\begin{definition}
A function $f \colon \rquo{G/K} \to \RR$ is \emph{geodesically convex} if for all $g \in G$ and $H \in \ii \Lie(K)$, the map $t \mapsto f(Ke^{tH}g)$ is convex in $t \in \RR$.
\end{definition}


\begin{example}[The Kempf-Ness function]
\label{ex:Kempf-Ness}
A fundamental example of a geodesically convex function arises from the log-norm function.
Let $\pi \colon G \to \GL(V)$ be a rational representation, and fix a $K$-invariant Hermitian inner product on $V$. For $v \neq 0$, the \emph{Kempf-Ness function} is:
 \[
	F_v(Kg) \coloneqq \log\Vert g\cdot v\Vert.
\] By \cite[Proposition~3.13]{BFGOWW-19}, $F_v$ is geodesically convex.
\end{example}

\begin{definition}[Gradient and Hessian]
    Let $f \colon \rquo{G/K} \to \RR$ be a smooth function. 
    The gradient of $f$ at $Kg$, denoted $\nabla f(Kg)$, is the unique tangent vector in $\ii\Lie(K)$ satisfying 
    \begin{equation}
        \label{eq:grad-def}
        \langle \nabla f(Kg), H \rangle = \left. \frac{\mathrm{d}}{\mathrm{d}t} \right|_{t=0} f\big(K e^{tH} g\big) \quad \text{for all } H \in \ii\Lie(K),
    \end{equation}
    where the inner product is the Frobenius inner product \eqref{eq:Frob-inner}.
    
    The Hessian $\nabla^2 f(Kg)$ of $f$ at $Kg$ is the linear operator $\nabla^2 f(Kg):\ii\Lie(K)\rightarrow\ii\Lie(K)$ defined by  \begin{equation}
        \label{eq:def-Hess}
        \langle H, \nabla^2 f(Kg) H \rangle = \left. \frac{\mathrm{d}^2}{\mathrm{d}t^2} \right|_{t=0} f\left(Ke^{tH}g\right),
    \end{equation} for all $H\in\ii\Lie(K)$.
\end{definition}
We will use the formulas for the gradient and the Hessian of the Kempf-Ness function.
\begin{proposition}
\label{prop:grad-Hess-KN}
Suppose $\pi:G\rightarrow\GL(V)$ is a rational representation with corresponding Lie algebra representation $\Pi:\Lie(G)\rightarrow\mathrm{End}(V)$.
Let $v\in V$. 
Then the gradient and the Hessian of the Kempf-Ness function $F_v$ (\cref{ex:Kempf-Ness}) are given by the formulas \begin{equation}
\label{eq:KN-grad-Hess-formulas}
	\begin{split}
	\langle \nabla F_v(Kg), H\rangle &= \langle \Pi(H) w, w\rangle \text{ and } \\
	\frac{1}{2}\langle \nabla^2 F_v(Kg)H, H\rangle &= \Vert \Pi(H)w\Vert^2 - \langle \Pi(H) w,w\rangle^2 , 
	\end{split}
	\end{equation} where $w\coloneqq \frac{g\cdot v}{\Vert g\cdot v\Vert}$.
\end{proposition}
\begin{proof}
Let $H\in\ii\Lie(K)$ and denote $u(t)\coloneqq e^{tH}g\cdot v$. 
Then, $u(0) = g\cdot v$ and $w = \frac{u(0)}{\Vert u(0)\Vert}$.
By \eqref{eq:Lie-alg-action}, we have $u'(t)=\Pi(H)u(t)$.

We write $F_v(Ke^{tH}g) = \frac{1}{2}\log\Vert u(t)\Vert^2$. 
The chain rule yields \[
\frac{\mathrm{d}}{\mathrm{d}t} F_v\left(Ke^{tH}g\right) = \frac{1}{2}\frac{\langle u(t),u'(t)\rangle + \langle u'(t),u(t)\rangle}{\Vert u(t)\Vert^2} = \frac{\langle \Pi(H)u(t),u(t)\rangle }{\Vert u(t)\Vert^2}.
\] Evaluating at $t=0$ gives $\langle \nabla F_v(Kg),H\rangle = \langle \Pi(H)w,w\rangle$.

For the Hessian, we take another derivative: \[
\frac{\mathrm{d}^2}{\mathrm{d}t^2} F_v\left(Ke^{tH}g\right) = \frac{2\, \langle \Pi(H)^2 u(t), u(t)\rangle\, \Vert u(t)\Vert^2 - 2 \,\langle \Pi(H)u(t),u(t)\rangle^2}{\Vert u(t)\Vert^4}
\] Evaluating at $t=0$ and using $w=u(0)/\Vert u(0)\Vert$ yields \eqref{eq:KN-grad-Hess-formulas}.
\end{proof}

\begin{remark}
Geodesic convexity of a smooth function $f$ is equivalent to $\nabla^2 f(Kg) \succeq 0$ for all $Kg \in \rquo{G/K}$.
For the Kempf-Ness function this follows from an application of the Cauchy-Schwarz inequality to \eqref{eq:KN-grad-Hess-formulas}. 
For optimization, we further quantify smoothness:
\end{remark}

\begin{definition}[$L$-smoothness]
\label{def:L-smooth}
A geodesically convex function $f$ is \emph{$L$-smooth} if $\nabla^2 f(Kg) \preceq L \cdot I$ for all $Kg \in \rquo{G/K}$, i.e., its Hessian's largest eigenvalue is bounded by $L$.
\end{definition}

The Taylor expansion of $f(Ke^{tH}g)$ yields a key inequality for $L$-smooth functions: see for example~\cite[Lemma~3.8]{BFGOWW-19}.
\begin{lemma}
\label{lem:Taylor-smoothness}
    If $f$ is $L$-smooth and $Kh = K e^{H} g$, then:
    \[
    f(Kh) \leq f(Kg) + \langle \nabla f(Kg), H \rangle + \frac{L}{2} \Vert H \Vert_F^2.
    \]
\end{lemma}

For an $L$-smooth geodesically convex function $f$, we define the \emph{smooth gradient descent} with starting point $Kg_0$ as the following iteration: \begin{equation}
	\label{eq:smooth-grad-desc}
	K g_{i+1} = K \, \exp\left( -\frac{1}{L}\, \nabla f(Kg_i) \right)\, g_i,\quad g_0=I_m,
\end{equation} where $\exp$ denotes the matrix exponential.

\cref{lem:Taylor-smoothness} implies the convergence of the smooth gradient descent \cite[Theorem~4.1]{BFGOWW-19}; we provide a proof for completeness.

\begin{proposition}[Iteration complexity]
\label{prop:grad-eps}
    Let $f \colon \rquo{G/K} \rightarrow \RR$ be geodesically convex and $L$-smooth. 
    Let $f_\star \coloneqq \inf_{Kg \in \rquo{G/K}} f(Kg)$. 
    Then, for any $\varepsilon > 0$, the gradient descent iteration in \eqref{eq:smooth-grad-desc}, starting from $g_0 = I$, after $T$ iterations, results a point $g_T$ satisfying $\Vert \nabla f(Kg_T) \Vert \leq \varepsilon$, where 
    \[
    T = \frac{2L}{\varepsilon^2} (f(I) - f_\star) .
    \]
\end{proposition}

\begin{proof}
    Assume $\Vert \nabla f(Kg_i) \Vert > \varepsilon$ for all $i = 0,\dots,T$. 
    From Lemma \ref{lem:Taylor-smoothness}, we have:
    \[
    \begin{split}
        f(Kg_{i+1}) - f(Kg_i) &\leq -\frac{1}{L} \Vert \nabla f(Kg_i) \Vert^2 + \frac{1}{2L} \Vert \nabla f(Kg_i) \Vert^2 \\
        &= -\frac{1}{2L} \Vert \nabla f(Kg_i) \Vert^2 \\
        &< -\frac{\varepsilon^2}{2L}.
    \end{split}
    \]
    Summing over $i = 0,\dots,T-1$ yields:
    \[
    f(Kg_T) - f(Kg_0) < -\frac{\varepsilon^2 T}{2L} \leq f_\star - f(Kg_0),
    \]
    which implies $f(Kg_T) < f_\star$ - a contradiction. Thus, $\Vert \nabla f(Kg_t) \Vert \leq \varepsilon$ for some $t \leq T$. By \cite[Lemma~3.8]{Hirai-Sakabe-25}, the gradient norm is non-increasing, so $\Vert \nabla f(Kg_T) \Vert \leq \Vert \nabla f(Kg_t) \Vert \leq \varepsilon$ as claimed.
\end{proof}

\subsection{Non-commutative duality for the log-norm function}
Suppose $f:\rquo{G/K}\rightarrow\RR$ is a geodesically convex function, and $f_\star = \inf_{g\in G} f(Kg)$.
The convexity implies that \[
f(Kg) = f_\star \quad\iff\quad \nabla f(Kg)=0.
\] In \cite{BFGOWW-19}, the authors gave a quantitative generalization of this equivalence for the Kempf-Ness function $F_v$, associated to a representation $\pi:G\rightarrow\GL(V)$ and $0\neq v\in V$.

\begin{definition}
\label{def:weight-norm-margin}
Assume that $T_m(\CC)\leq G\leq\GL_m(\CC)$ is symmetric and $\pi:G\rightarrow\GL(V)$ is a rational representation with weights $\Omega(\pi)=\{\omega_1,\omega_2,\dots,\omega_N\}\subset\ZZ^m$. 
	\begin{itemize}
		\item The \emph{weight norm} of $\pi$ is defined as \[
		N(\pi)\coloneqq \max_{\omega\in\Omega(\pi)} \Vert \omega\Vert_2.
		\]
		\item The \emph{weight margin} of $\pi$ is defined as \[
    \gamma(\pi) \coloneqq \min \left\{ 
        \mathrm{dist}\big(0, \conv(S)\big) \mid S \subset \Omega(\pi), \, 0 \notin \conv(S) 
    \right\},
    \]
    where $\mathrm{dist}(0, \conv(S)) = \min \{\|p\|_2 \mid p \in \conv(S)\}$ is the Euclidean distance from the origin to $\conv(S)$.
	\end{itemize}
\end{definition}
The weight norm and the weight margin quantifies the proximity to the optimal value in terms of the gradient norm:
\begin{theorem}[Non-commutative duality]
\label{thm:nc-duality}
	Let $\pi:G\rightarrow\GL(V)$ be a rational representation of a symmetric subgroup $T_m(\CC)\leq G\leq \GL_m(\CC)$. 
	Let $0\neq v\in V$ and $F_v(Kg)=\log\Vert g\cdot v\Vert$ be the Kempf-Ness function associated to $v$. 
	Then, \[
	1 - \frac{\Vert\nabla F_v(Kg)\Vert_F}{\gamma(\pi)} \, \leq \, \frac{ \inf_{h\in G} \, \Vert h\cdot v\Vert^2 }{ \Vert g\cdot v\Vert^2 } \, \leq \,1 - \frac{\Vert\nabla F_v(Kg)\Vert^2_F}{4N(\pi)^2}. 	\]
\end{theorem}
\begin{proof}
	See \cite[Theorem~1.17]{BFGOWW-19}.
\end{proof}
Non-commutative duality shows that small gradient norm implies proximity to the optimal value.
The upper bound in \cref{thm:nc-duality} follows from the following bound on the smoothness of the Kempf-Ness function.
\begin{lemma}
\label{lem:KN-smoothness}
	Let $\pi:G\rightarrow\GL(V)$ be a rational representation of a symmetric subgroup $G\leq \GL_m(\CC)$. 
	The Kempf-Ness function $F_v$ associated to $0\neq v\in V$ is $2N(\pi)^2$-smooth.
\end{lemma}
\begin{proof}
    See \cite[Proposition~3.13]{BFGOWW-19}.
\end{proof}
An application of \cref{lem:Taylor-smoothness} to the Kempf-Ness function $F_v$ with $g=I_m$ and $H=-\frac{\nabla F_v(Kg)}{2N(\pi)^2}$ gives the upper bound in \cref{thm:nc-duality}, see \cite[Theorem~3.23]{BFGOWW-19}.

The lower bound in \cref{thm:nc-duality} seems to be much more interesting.
In general, it is not easy to estimate the weight margin $\gamma(\pi)$ given the weights. 
A general lower bound on the weight margin is given in \cite[Theorem~6.10]{BFGOWW-19}:
\begin{proposition}
\label{prop:margin-general-lb}
	Suppose $G\leq\GL_{m_1}(\CC)\times\dots\times\GL_{m_k}(\CC)$ and $T_{m_1}(\CC)\times\dots\times T_{m_k}(\CC)\leq G$. If $\pi:G\rightarrow\GL(V)$ is a rational representation of $G$,
	then $
	\gamma(\pi)\geq N(\pi)^{1-m}\, m^{-1}
	$ where $m\coloneqq m_1+\dots+m_k$ and $N(\pi)$ is the weight norm of $\pi$. 
\end{proposition}

For general representations, \cref{prop:margin-general-lb} seems to be sharp. 
In special cases, it is possible to give better lower bounds, such as when the weight matrix $M$ is \emph{totally unimodular}. 
Recall that an integer matrix $M$ is called totally unimodular if all of its subdeterminants are $0,\pm 1$.

\begin{proposition}[{\cite[Corollary~6.11]{BFGOWW-19}}]
\label{prop:tu-margin}
	Let $G\leq\GL_m(\CC)$ be a symmetric subgroup that contains $T_m(\CC)$.
	Suppose $\pi:G\rightarrow\GL(V)$ is a rational representation
	 with weight matrix $M\in\ZZ^{m\times N}$.
	If $M$ is totally unimodular, then \[
	\gamma(\pi) \geq m^{-\frac{3}{2}}.
	\]
\end{proposition}

\subsection{Strongly convex optimization}
We now consider the class of strongly convex functions on manifolds. 
This subject seems to be overlooked in \cite{BFGOWW-19}: 
This is due to fact that the Kempf-Ness function is in general not strongly convex.

\begin{definition}
\label{def:mu-strongly-convex}
A smooth geodesically convex function $f \colon \rquo{G/K} \rightarrow \RR$ is called $\mu$-strongly convex ($\mu > 0$) on a sublevel set $S\coloneqq\{Kg\in \rquo{G/K}\mid f(Kg)\leq f(Kg_0)\}$ if
\[
\langle H, \nabla^2 f(Kg) H \rangle \geq \mu \Vert H \Vert_F^2
\]
for all $Kg \in S$ and $H \in \ii\Lie(K)$. 
Equivalently, the smallest eigenvalue of $\nabla^2 f(Kg)$ is at least $\mu$.
\end{definition}
This condition implies that the function grows at least quadratically in all directions, which is crucial for establishing fast convergence rates.
The gradient descent algorithm for strongly convex functions exhibit linear convergence rates, giving rise to $O(\log\tfrac{1}{\varepsilon})$-time algorithms.

\begin{remark}
    Strongly convex functions enjoy a similar inequality as \cref{lem:Taylor-smoothness}:
    $f(K e^H g) \geq f(Kg) + \langle \nabla f(Kg), H\rangle + \frac{\mu}{2}\Vert H\Vert_F^2$.
    When specialized to $e^{H}g=\arg\min_{Kg} f(Kg)$, we obtain the Polyak-\L ojasiewicz inequality: \[
    f(Kg)-f_\star \leq \frac{1}{2\mu}\,\Vert \nabla f(Kg)\Vert^2,
    \] see \cite[Lemma~11.28]{Boumal-23} for the details. 
    We note in particular that the Polyak-\L ojasiewicz inequality quantifies proximity to the optimal value in terms of the gradient bound and serves a similar role as the lower bound in non-commutative duality \cref{thm:nc-duality}.
\end{remark}
The convergence of the smooth gradient descent for strongly convex functions on Euclidean spaces is well-known \cite{Boyd-Vandenberghe-11}.
The iteration complexity is generally scales with $\frac{L}{\mu}$, the condition number of the Hessian.
The following is a geodesically convex generalization, see \cite[Theorem~11.29]{Boumal-23}.
\begin{theorem}
\label{thm:strongly-convex-grad-descent}
    If $f \colon \rquo{G/K} \rightarrow \RR$ is $L$-smooth and $\mu$-strongly convex on $S \coloneqq \{Kg \mid f(Kg) \leq f(Kg_0)\}$, then the iteration \eqref{eq:smooth-grad-desc} with initial point $g_0$ satisfies
\begin{equation}
\label{eq:strongly-convex-linear-conv}
f(Kg_T) - f_\star \leq \left(1 - \frac{\mu}{L}\right)^T (f(Kg_0) - f_\star)
\end{equation}
\end{theorem}
\begin{corollary}
\label{cor:strongly-convex-descent}
	With the assumptions of \cref{thm:strongly-convex-grad-descent} and $\kappa\coloneqq \frac{L}{\mu}$, we have $f(Kg_T)-f_{\star}<\varepsilon$ for \[
	T = \,\kappa\,\log\left(\frac{f(Kg_0)-f_\star}{\varepsilon}\right).
	\] 
\end{corollary}
\begin{proof}
	Taking the logarithm of \eqref{eq:strongly-convex-linear-conv} and using $\log(1-\kappa^{-1})<-\kappa^{-1}$ (which holds since $\kappa\geq 1$), \[
	\log\left( f(Kg_T)-f_\star \right) < -\frac{T}{\kappa} + \log\left( f(Kg_0)-f_\star\right) \leq \log(\varepsilon).
	\] Hence, $f(Kg_T)-f_\star <\varepsilon$, establishing the claim.
\end{proof}

\section{Preconditioning Linear Systems}
\label{sec:linear-systems}
In this section we consider preconditioning matrices.
We will prove \crefpart{thm:main-g-convexity}{thm:main-part-1}, and \cref{thm:general-linear-complexity,thm:main-linear-algorithm}.

\subsection{Convexity of log-condition} 
\label{sec:setup}
We establish the geodesic convexity of the log-condition number function $C_A(K(X,Y))=\log\kappa(XAY^{-1})$ (defined in \cref{eq:C_A-def}) for arbitrary unitarily invariant norms as a direct consequence of the Heinz-Kato inequality \cite{Heinz-51, Kato-61}. 
The following version of the Heinz-Kato inequality, stated for an arbitrary unitarily invariant matrix norm, is due to Kittaneh \cite[Theorem~1]{Kittaneh-93} and Bhatia and Davis \cite{Bhatia-Davis-95}.
\begin{lemma}[The Heinz-Kato inequality]
\label{lem:heinz-kato}
    Suppose $A\in\CC^{m\times n}$ and $\Vert\cdot\Vert$ is a unitarily invariant matrix norm on $\CC^{m\times n}$. Then, for any pair of positive definite matrices $X\in\PD(m)$ and $Y\in\PD(n)$, \[
    \Vert \sqrt{X} A \sqrt{Y}\Vert^2 \leq \Vert A\Vert\,\Vert XAY\Vert.
    \]
\end{lemma}
We recall that every positive definite matrix $X$ admits a unique positive definite square root $\sqrt{X}$, i.e., a positive definite matrix that satisfies $(\sqrt{X})^2 = X$.
For a Hermitian matrix $H$, one easily confirms $\sqrt{e^{tH}}=e^{\frac{t}{2}H}$.
Hence, the Heinz-Kato inequality implies that the function $f(t)\coloneqq \Vert e^{tH_1} A e^{tH_2}\Vert$ satisfies $f^2(t/2) \leq f(0) f(t)$ for positive $t\in\RR_{>0}$. 
This will be enough to prove the log-convexity of $f(t)$.
\begin{proposition}
\label{prop:convexity-of-log-norm}
   Suppose $B\in\CC^{m\times n}$ and $\Vert\cdot\Vert$ is a unitarily invariant matrix norm on $\CC^{m\times n}$.
    Then, for any pair of Hermitian matrices $H_1\in\Herm(m)$ and $H_2\in\Herm(n)$, the following function is convex: \begin{equation}
    	\label{eq:log-norm-h}
    h(t)= \log\Vert e^{tH_1} B e^{tH_2}\Vert.
    \end{equation}
\end{proposition}
\begin{proof}
    Since $\log\Vert e^{tH_1} B e^{tH_2}\Vert$ is continuous, is suffices to prove the mid-point convexity, i.e., \begin{equation}
    \label{eq:log-convexity-prop}
    \log\Vert e^{\frac{s+t}{2}\, H_1} B e^{\frac{s+t}{2}\, H_2}\Vert \leq \frac{1}{2}\Big(\, \log\Vert e^{sH_1} Ae^{sH_2}\Vert \, + \,\log \Vert e^{tH_1} B e^{tH_2} \Vert\, \Big)
    \end{equation} for arbitrary $s<t$. 
    We rearrange \begin{equation}
    \label{eq:rewrite}
     e^{\frac{s+t}{2}\, H_1} B e^{\frac{s+t}{2}\, H_2} =  e^{\frac{t-s}{2}H_1}\,\big(\, e^{sH_1} B e^{sH_2}\,\big)\, e^{\frac{t-s}{2}H_2}.
    \end{equation}
    Noting that $e^{\frac{t-s}{2}H_1}=\sqrt{e^{(t-s)H_1}}$ (and similarly for $H_2$), 
    we apply the Heinz-Kato inequality to the right-hand side of \cref{eq:rewrite}, yielding \[
    \Vert  e^{\frac{s+t}{2}\, H_1} B e^{\frac{s+t}{2}\, H_2}\Vert^2 \leq \Vert e^{sH_1} B e^{sH_2}\Vert \, \Vert e^{(t-s)H_1} \, e^{sH_1} B e^{ sH_2}  e^{(t-s)H_2}\Vert =  \Vert e^{sH_1}B e^{sH_2} \Vert \, \Vert e^{tH_1} B e^{tH_2}\Vert.
    \] This inequality is equivalent to \cref{eq:log-convexity-prop}, completing the proof.
\end{proof}
We now prove \crefpart{thm:main-g-convexity}{thm:main-part-1}.
\begin{proof}[Proof of \crefpart{thm:main-g-convexity}{thm:main-part-1}]
    We set $B\coloneqq XAY^{-1}$ and $B^\dag= YA^{\dag}X^{-1}$.
    Then we express $C_A(e^{tH_1}X, e^{tH_2}Y)$ as \[ C_A(e^{tH_1}X, e^{tH_2}Y) = \log\Vert e^{tH_1} B e^{-tH_2}\Vert + \log\Vert e^{tH_2} B^\dag e^{-tH_1}\Vert.\] 
    By \cref{prop:convexity-of-log-norm}, both terms are convex in $t$.
    Thus, $C_A(e^{tH_1}X, e^{tH_2}Y)$ is convex in $t$, and by \cref{def:geodesic-convexity} of geodesic convexity, $C_A$ is geodesically convex on $\rquo{G/K}$.
\end{proof}

\begin{remark}
    The proof of \cref{thm:main-g-convexity} shows even more generally that the function \[
    C_{A,B}(K(X,Y)) \coloneqq \log\kappa(XAY^{-1}, YBX^{-1})
    \] is geodesically convex, where $\kappa(A,B)\coloneqq \Vert A\Vert_1 \Vert B\Vert_2$ is the \emph{cross condition number} induced by $\Vert\cdot\Vert_1$ and $\Vert\cdot\Vert_2$.
    The convexity of $\log\kappa(XAY^{-1})$ is then a special case as $\kappa(A,A^{\dag})=\kappa(A)$.
\end{remark}

An important case that we will deal in the next chapters is when both $\CC^{m\times n}$ and $\CC^{n\times m}$ are endowed with the Frobenius norm.
In this case, the geodesic convexity of $C_A\left(K(X,Y)\right)$ follows from \cref{ex:Kempf-Ness}:
We may write \[
\kappa_F(A)\coloneqq\Vert A\Vert_F \Vert A^\dag\Vert_F = \Vert A\otimes A^\dag\Vert
\] where the latter norm on $\CC^{m\times n}\otimes\CC^{n\times m}$ is induced from the inner product on tensor product, defined by $\langle A\otimes B,C\otimes D\rangle=\tr(A^* C)\tr(B^* D)$ and linearly extending to the whole tensor product. 
Hence, optimization of $\kappa_F(A)$ is equivalent to the optimization of the Kempf-Ness function associated to $A\otimes A^{\dag}$, where the action of $\GL_m(\CC)\times\GL_n(\CC)$ on $\CC^{m\times n}\otimes\CC^{n\times m}$ is given by \[
(X,Y)\cdot( A\otimes B )= XAY^{-1} \otimes YBX^{-1}.
\]
Consequently, the cross condition number $\log\kappa_F(XAY^{-1},YBX^{-1})$ becomes the Kempf-Ness function associated to the vector $A\otimes B$, and when specialized to $B=A^\dag$, $\log\kappa_F(XAY^{-1})$ becomes the Kempf-Ness function associated to $A\otimes A^\dag$.
Hence, it is geodesically convex on $\rquo{G/K}$ by \cref{ex:Kempf-Ness}.

\subsection{The gradient of the log-condition}
In this section, we will analyse the parameters governing the complexity of the gradient descent for the log-condition number. 
Specifically, we seek to bound the smoothness parameter of $C_A$ (\cref{def:L-smooth}).
This is only possible if $C_A$ is smooth, i.e., when the matrix norm is a smooth function.
For instance, the operator norm $\Vert \diag(1+t, 1-t)\Vert_{op} = |t|$ for $t \in (-1,1)$ fails to be differentiable at $t=0$. 
To circumvent this, we instead work with the Frobenius norm $\Vert A\Vert_F^2 = \tr(A^* A)$, which is smooth and induces the Frobenius condition number/cross condition number:
\[
\kappa_F(A) \coloneqq \Vert A\Vert_F \Vert A^\dag\Vert_F,\quad \kappa_F(A,B)\coloneqq \Vert A\Vert_F \, \Vert B\Vert_F.
\]
We will study the log-condition/cross-condition function 
\[
C_A\big(K(X,Y)\big) \coloneqq \log \kappa_F(XAY^{-1}),\quad C_{A,B}\big(K(X,Y)\big)\coloneqq \log\kappa_F(XAY^{-1}, YBX^{-1}),
\]
defined on the symmetric space $\rquo{G/K}$, where $G \leq \GL_m(\CC) \times \GL_n(\CC)$ and $K=G\cap \rmU_m\times \rmU_n$ is its maximal compact subgroup.
Clearly, $C_A = C_{A,A^{\dag}}$.

We note that $C_{A,B}(K(X,Y))$ is the Kempf-Ness function associated to $A\otimes B$ (\cref{ex:Kempf-Ness}), where the action of $\GL_m(\CC)\times\GL_n(\CC)$ on $\CC^{m\times n}\otimes \CC^{n\times m}$ is given by \[
\pi(X,Y) ( A\otimes B ) = XAY^{-1} \otimes YBX^{-1}.
\] That is, \[
C_{A,B}(K(X,Y)) = \log\Vert \pi(X,Y) (A\otimes B)\Vert, 
\] where the norm on $\CC^{m\times n}\otimes\CC^{n\times m}$ is given by $\Vert A\otimes B\Vert=\Vert A\Vert_F\Vert B\Vert_F$. 

\begin{lemma}
\label{lem:preco-Lie-action}
	For $H=(H_1,H_2)\in\ii\Lie(K)$, the Lie algebra action of $H$ on $A\otimes B$ is given by \[
	\Pi(H) \left(A\otimes B\right) = \left(H_1 A - AH_2\right)\otimes B + A \otimes \left( H_2 B - BH_1 \right).
	\]
\end{lemma}
\begin{proof}
	We have \[
	\Pi(H) A\otimes B = \left. \frac{\mathrm{d}}{\mathrm{d}t} \right|_{t=0} \pi(e^{tH_1},e^{tH_2})\, A\otimes B = \left. \frac{\mathrm{d}}{\mathrm{d}t} \right|_{t=0} e^{tH_1} A e^{-tH_2}\otimes e^{tH_2}B e^{-tH_1}.
	\]
	Using $\frac{\mathrm{d}}{\mathrm{d}t}e^{tH}= He^{tH}$ and the product rule, we get \[
	\left. \frac{\mathrm{d}}{\mathrm{d}t} \right|_{t=0} e^{tH_1} A e^{-tH_2}\otimes e^{tH_2}B e^{-tH_1} = \left(H_1 A - AH_2\right)\otimes B + A \otimes \left( H_2 B - BH_1 \right).
	\]
\end{proof}

\begin{proof}[Proof of \cref{lem:gradient}]
	\cref{prop:grad-Hess-KN} gives \[
	\langle \nabla C_A(K(X,Y)) , H\rangle = \langle \Pi(H) E , E\rangle
	\] where $E =(X,Y)\cdot\frac{ A\otimes A^{\dag}}{\Vert A\otimes A^\dag\Vert}$. 
	Denote $B\coloneqq XAY^{-1}$.
	By \cref{lem:preco-Lie-action}, we have \[
	\begin{split}
	\langle \Pi(H) E , E\rangle &= \frac{\langle (H_1 B - BH_2) \otimes B^\dag + B \otimes \left( H_2 B^\dag - B^\dag H_1 \right) , B\otimes B^{\dag} \rangle }{\Vert B\otimes B^{\dag}\Vert^2}\\
	&= \frac{ \tr\left(\, (H_1 B-BH_2)^\ast B \, \right)\times \Vert B^\dag\Vert_F^2 + \tr\left(\, (H_2 B^\dag-B^\dag H_1)^\ast B^\dag \, \right)\times \Vert B\Vert_F^2 }{ \Vert B\Vert_F^2 \Vert B^\dag\Vert_F^2 }\\
	&=\frac{\tr( H_1 B B^\ast)}{\Vert B\Vert_F^2} - \frac{\tr(H_2 B^\ast B)}{\Vert B\Vert_F^2} - \frac{\tr(H_1 (B^{\dag})^\ast B^\dag)}{\Vert B^\dag\Vert_F^2} + \frac{\tr(H_2 B^\dag (B^\dag)^\ast)}{\Vert B^\dag\Vert_F^2}\\
	&= \tr(H_1 P ) + \tr(H_2 Q)
	\end{split}
	\] where \[
	P \coloneqq \frac{B B^*}{\Vert B \Vert_F^2} - \frac{(B^\dag)^* B^\dag}{\Vert B^\dag \Vert_F^2}, \quad
    Q \coloneqq -\frac{B^* B}{\Vert B \Vert_F^2} + \frac{B^\dag (B^\dag)^*}{\Vert B^\dag \Vert_F^2}.
	\]
	
    By the definition of the inner product \eqref{eq:Frobenius-inner} on $\Lie(G)$, we have
    \[
     \langle H,\nabla C_A(K(X,Y))\rangle  = \tr(H_1 \nabla C_A(K(X,Y))_1) + \tr(H_2 \nabla C_A(K(X,Y))_2)
    \]
    for all $H = (H_1, H_2) \in \ii \Lie(K)$. 
    Here, $\nabla C_A(K(X,Y))_i$ denotes the projection onto the $i$-th component of $\Lie(G)\subset\CC^{m\times m}\oplus\CC^{n\times n}$. 
    Comparing with above, we deduce that
    \[
    \nabla C_A(K(X,Y)) = \proj_{\Lie(G)} \left( P, Q \right),
    \]
    where $\proj_{\Lie(G)}$ is the orthogonal projection with respect to the inner product \eqref{eq:Frobenius-inner}.
\end{proof}

\subsection{Weights for preconditioning}
\subsubsection*{Left preconditioning}
We first specialize to the action of $G\leq \GL_m(\CC)$ and consider only left preconditioners.
We start by computing the weights for the representation of $G$ on $\CC^{m\times n}\otimes\CC^{n\times m}$ via \begin{equation}
\label{eq:left-preco-tensor}
	X \cdot ( A\otimes B ) = XA \otimes BX^{-1}.
\end{equation}
We will call this action \emph{left preconditioning}.

\begin{lemma}[Weights of the left preconditioning action]
\label{lem:left-prec-weights}
    Let $G \leq \GL_m(\CC)$ be a symmetric subgroup containing $T_m(\CC)$. 
    The weights of the action \eqref{eq:left-preco-tensor} are
    \[
    \omega_{ij} = e_i - e_j \quad \text{for } 1 \leq i, j \leq m,
    \]
    where each $\omega_{ij}$ is repeated $n^2$ times and $\{e_i\}$ are the standard basis vectors in $\ZZ^m$.
\end{lemma}
\begin{proof}
    The elementary matrices $E_{ik} \otimes E_{lj}$ form a basis for $\CC^{m \times n} \otimes \CC^{n \times m}$. 
    For $X = \diag(t_1, \dots, t_m) \in T_m(\CC)$, we compute:
    \[
    X \cdot (E_{ik} \otimes E_{lj}) = (XE_{ik})\otimes (E_{lj}X^{-1}) = (t_i E_{ik}) \otimes (E_{lj} t_j^{-1}) = (t_i t_j^{-1}) E_{ik} \otimes E_{lj}.
    \]
    Thus, $E_{ik} \otimes E_{lj}, 1\leq k,l\leq n$ is a simultaneous eigenvector of $T_m(\CC)$ with weight $e_i - e_j$.
\end{proof}

\cref{lem:left-prec-weights} shows that the weight matrix $M_l\in\ZZ^{m^2 n^2\times m}$ of left-preconditioning is of the form \[
M_l \coloneqq \begin{pmatrix}
	& \vdots & \\
	& e_i - e_j & \\
	& \vdots &
\end{pmatrix}
\] where $1\leq i,j\leq m$ and the row $e_i-e_j$ is repeated $n^2$ times.

\begin{lemma}
\label{lem:directed-TU}
	$M_l$ is totally unimodular.
\end{lemma}
\begin{proof}
	We note that $(M_l)^T$ is the incidence matrix of the complete directed graph $K_n$.
	By a standard argument, $M_l$ is totally unimodular, see e.g. \cite[Section~19, Example~2]{Schrijver-99}.
\end{proof}

We can lower bound the weight margin of left-preconditioning using total unimodularity.
\begin{corollary}
	The weight norm and the weight margin (\cref{def:weight-norm-margin}) for left-preconditioning satisfies \[
	N_l=\sqrt{2},\quad \gamma_l \geq m^{-\frac{3}{2}}.
	\]
\end{corollary}
\begin{proof}
	$N_l=\sqrt{2}$ follows from $\Vert e_i-e_j\Vert=\sqrt{2}$.
	The bound for the weight margin follows from \cref{prop:tu-margin,lem:directed-TU}.
\end{proof}

\cref{lem:KN-smoothness,thm:nc-duality} yields:
\begin{corollary}
\label{cor:left-preco-duality}
Let $A\in\CC^{m\times n},B\in\CC^{n\times m}$ and denote \[
C_{A,B}(KX)\coloneqq \log\Vert XA\Vert_F + \log \Vert BX^{-1}\Vert_F,
\] the logarithm of the cross condition number $\kappa_F(XA,BX^{-1})$.
	
Then $C_{A,B}(KX)$ is $4$-smooth.
	Moreover, for $X\in G$, \[
	\left(\frac{\kappa_F(XA,BX^{-1})}{\kappa_F^\star}\right)^2 \, \leq \, \left( 1 - m^{\frac{3}{2}}\, \Vert\nabla C_{A,B}(KX)\Vert \right)^{-1},
	\] where $\kappa_F^{\star}\coloneqq \min_{X\in G}\kappa_F(XA,BX^{-1})$.
\end{corollary}

\subsubsection*{Left-right preconditioning}
We now compute the weights for left-right preconditioning.
\begin{lemma}[Weights for left-right preconditioning]
    Let $G \leq \GL_m(\CC) \times \GL_n(\CC)$ be a symmetric subgroup containing $T_m(\CC) \times T_n(\CC)$. 
    Under the action 
    \[
    (X,Y) \cdot (A \otimes B) = XAY^{-1} \otimes YBX^{-1}
    \]
    on $\CC^{m \times n} \otimes \CC^{n \times m}$, the weights are
    \[
    \omega_{ijkl} \coloneqq (e_i - e_j, e_k - e_l), \quad 1 \leq i,j \leq m, \; 1 \leq k,l \leq n,
    \]
    where $\{e_i\}_{i=1}^m$ and $\{e_k\}_{k=1}^n$ are the standard basis vectors for $\ZZ^m$ and $\ZZ^n$ respectively.
\end{lemma}
\begin{proof}
    The elementary matrices $E_{ik} \otimes E_{lj}$ form a basis of $\CC^{m\times n}\otimes\CC^{n\times m}$. 
    For a pair of diagonal matrices, $X = \diag(t_1,\dots,t_m)$ and $Y=\diag(s_1,\dots,s_n))$, we compute:
    \[
    (X,Y) \cdot (E_{ik} \otimes E_{lj}) = (t_i s_k^{-1}) E_{ik} \otimes (s_l t_j^{-1}) E_{lj} = (t_i t_j^{-1} s_k^{-1} s_l) E_{ik} \otimes E_{lj}.
    \]
    Thus, each $E_{ik} \otimes E_{lj}$ is an eigenvector with weight $(e_i - e_j, e_k - e_l)\in\ZZ^m\times\ZZ^n$.
\end{proof}

Hence, the weight matrix $M_{lr}\in\ZZ^{m^2n^2\times (m+n)}$ of the left-right preconditioning is of the form \[
M_{lr} \coloneqq \begin{pmatrix}
	& \vdots & \vline & \vdots & \\
	& e_i - e_j & \vline & e_k-e_l & \\
	& \vdots & \vline & \vdots & \\
\end{pmatrix}
\]

\begin{lemma}
\label{lem:margin-for-lr}
	$M_{lr}$ is totally unimodular.
\end{lemma}
\begin{proof}
	We can write $M=\begin{pmatrix}
		M_l & \vline & M_r
	\end{pmatrix} \in\ZZ^{m^2 n^2\times (m+n)}$, where $M_l\in\ZZ^{m^2n^2\times m}$ has rows $e_i-e_j$ and $M_r\in\ZZ^{m^2n^2\times n}$ has rows $e_k-e_l$. 
	Assume without loss of generality that $m\geq n$.
	Let $N\in\ZZ^{m\times m}$ be the  
	Then, $M_r$ is a submatrix of $M_l$. 
	Hence, every submatrix of $M$ can be identified with a submatrix of $M_l$ with possibly repeated rows. 
	By \cref{lem:directed-TU}, $M_l$ is totally unimodular. 
	This implies that every submatrix of $M$ has determinant $0,\pm 1$. 
\end{proof}

\begin{corollary}
	The weight norm and the weight margin for the left-right preconditioning satisfies \[
	N_{lr} = 2, \quad \gamma_{lr}\geq (m+n)^{-\frac{3}{2}}.
	\]
\end{corollary}
\begin{proof}
	$N_{lr}=2$ follows from $\Vert (e_i-e_j) \oplus (e_k-e_l)\Vert=2$. 
	For the weight margin bound we use \cref{prop:tu-margin,lem:margin-for-lr}.
\end{proof}

By \cref{lem:KN-smoothness,thm:nc-duality} we get:

\begin{corollary}
\label{cor:left-right-preco-duality}
Let $A\in\CC^{m\times n},B\in\CC^{n\times m}$ and denote \[
C_{A,B}(K(X,Y))\coloneqq \log\Vert (X,Y)\cdot (A\otimes B)\Vert= \log\Vert XAY^{-1}\Vert_F + \log \Vert YBX^{-1}\Vert_F,
\] the logarithm of the cross condition number $\kappa_F(XAY^{-1},YBX^{-1})$.
	
Then $C_{A,B}$ is $8$-smooth.
	Moreover, for $(X,Y)\in G$, \[
	\left(\frac{\kappa_F(XAY^{-1},YBX^{-1})}{\kappa_F^\star}\right)^2 \, \leq \, \left( 1 - (m+n)^{\frac{3}{2}}\, \Vert\nabla C_{A,B}(K(X,Y))\Vert \right)^{-1},
	\] where $\kappa_F^{\star}\coloneqq \min_{X\in G}\kappa_F(XAY^{-1},YBX^{-1})$.
\end{corollary}

\subsection{Complexity analysis of smooth gradient descent} 
\label{sec:general-complexity-linear}
We now analyze the iteration complexity of smooth gradient descent for minimizing the log-cross condition number function $C_{A,B}$.

\begin{theorem}
\label{thm:cross-condition-iteration}
Let $G$ be either 
	(i) a symmetric subgroup of $\GL_m(\CC)\times\GL_n(\CC)$ with $T_m(\CC)\times T_n(\CC)\leq G$,
       where $T_m(\CC)$ is the group of diagonal elements,
       or (ii) a symmetric subgroup of $\GL_m(\CC)$ with $T_m(\CC)\leq G$. 
     
    Let $A\in\CC^{m\times n}$ and $B\in\CC^{n\times m}$ be matrices, denote by
    $\kappa_F^\star\coloneqq \inf_{(X,Y)\in G}\kappa_F(XAY^{-1},YBX^{-1})$
    the optimal cross condition number we can achieve by preconditioning $A$ and $B$ with an element $(X,Y) \in G$,
    and $\eps > 0$.
     
	There is a first-order (gradient descent) algorithm, with constant step-size $\eta=1/8$
	and the initial point $(X_0,Y_0)\coloneqq (I_m,I_n)$, 
	that reaches a group element $(X,Y)$, 
	such that $\log\kappa_F(XAY^{-1},YBX^{-1}) - \log\kappa_F^{\star}<\varepsilon$ in \[
	T = O\Big(\;  \frac{\log\left(\kappa_F(A,B)/\kappa_F^\star\right)}{\varepsilon^2}\, \max(m^3,n^3) \;\Big) 
		\quad \text{ iterations.}
	\]
\end{theorem}
\begin{proof}
    Set $\varepsilon' \coloneqq (m+n)^{-3/2} \varepsilon$. 
    By \cref{cor:left-preco-duality,cor:left-right-preco-duality},
    $C_{A,B}$ is an $8$-smooth function and \cref{prop:grad-eps} guarantees that gradient descent finds $(X,Y)$ with $\Vert \nabla C_{A,B}(K(X,Y)) \Vert \leq \varepsilon'$ in:
    \[
    T = O\left(\frac{\log(\kappa_F(A,B)/\kappa_F^\star)}{(\varepsilon')^2}\right) = O\left(\frac{(m+n)^3 \log(\kappa_F(A,B)/\kappa_F^\star)}{\varepsilon^2}\right)
    \]
    iterations. 
    \cref{cor:left-preco-duality,cor:left-right-preco-duality} give
    \[
   \left( \frac{\kappa_F^\star}{\kappa_F(XAY^{-1},YBX^{-1})}\right)^2  = \left(\frac{\min_{(X,Y)\in G} \Vert (X,Y)\cdot (A\otimes B)\Vert}{\Vert (X,Y)\cdot (A\otimes B)\Vert}\right)^2 \geq 1 - (m+n)^{3/2}\varepsilon' = 1 - \varepsilon.
    \]
    Taking logarithms and rearranging (for $\varepsilon < 0.5$) yields:
    \[
    2\log\left(\frac{\kappa_F(XAY^{-1},YBX^{-1})}{\kappa_F^\star}\right) \leq -\log(1-\varepsilon) < 2\varepsilon,
    \]
    completing the proof.
\end{proof}

\subsection{Strongly convex optimization}
\label{sec:strongly-convex-linear}
We now establish \cref{thm:main-linear-algorithm}.
The key insight is that the log-condition number for left-preconditioning becomes strongly convex when restricted to appropriate sublevel sets.

\begin{proposition}[Hessian bound for condition number]
\label{prop:strong-convexity-parameter}
Let $m \leq n$, $A \in \CC^{m \times n}$ be full rank ($\rk(A) = m$), and $G \leq \SL_m(\CC)$ a symmetric subgroup with $K \coloneqq G \cap \mathrm{SU}_m$. 
The Hessian of $C_A(KX) = \log \kappa_F(XA)$ satisfies:
\begin{equation}
\label{eq:Hess-left-mult}
\frac{1}{2}\langle H, \nabla^2 C_A(I)H \rangle = 
\frac{\tr(H^2 P)}{\tr(P)} - \frac{\tr^2(HP)}{\tr^2(P)} + 
\frac{\tr(H^2 P^{-1})}{\tr(P^{-1})} - \frac{\tr^2(HP^{-1})}{\tr^2(P^{-1})}
\end{equation}
where $P \coloneqq XAA^*X^*$. Moreover, on the sublevel set $S \coloneqq \{KX \in \rquo{G/K} \mid C_A(KX) \leq C_A(I)\}$, the smallest eigenvalue of $\nabla^2 C_A(KX)$ is bounded below by $4/\kappa_F^2(A)$.
\end{proposition}

\begin{remark}
The restriction to the sublevel set $S$ is natural in optimization, as we typically start from the identity matrix and seek improvement. The proposition shows that the landscape becomes increasingly favorable as we approach the optimal preconditioner.
\end{remark}

\begin{proof}
We first prove the formula for the Hessian. 
As $C_A(X)=\log\Vert XA\Vert_F + \log\Vert A^{\dag}X^{-1}\Vert_F$, we can compute the Hessian as the sum of the Hessian of each summand.
By \cref{prop:grad-Hess-KN}, the Hessian of the first summand at $X$ is given by \[
\frac{1}{2}\langle \nabla^2(\log\Vert XA\Vert_F) H, H \rangle = \frac{\tr\left( (H^2 B)^\ast  B\right)}{\Vert B\Vert_F^2} - \frac{\tr\left((HB)^\ast B\right)}{\Vert B\Vert_F^4},
\] where $B=XA$. 
Similarly, the Hessian of the second summand is \[
\frac{1}{2}\langle \nabla^2(\log\Vert AX^{-1}\Vert_F) H, H \rangle = \frac{\tr\left( (B^\dag H_2)^\ast  B^\dag\right)}{\Vert B^\dag\Vert_F^2} - \frac{\tr\left((B^\dag H)^\ast B^\dag\right)}{\Vert B^\dag\Vert_F^4}
\] by \cref{prop:grad-Hess-KN}. 
Setting $P=BB^\ast$, we obtain \eqref{eq:Hess-left-mult}.

We now prove the lower bound. 
Let $P = AA^*$ (positive definite since $\rk(A) = m$). 
By \cref{lem:strong-convexity-lb} below:
\[
\frac{\tr(H^2 P)}{\tr(P)} - \frac{\tr^2(HP)}{\tr^2(P)} \geq \frac{\Vert H \Vert_F^2}{\kappa_1(P)}
\]
where $\kappa_1(P) = \tr(P)\tr(P^{-1})$. Applying the same lemma to $P^{-1}$ yields:
\[
\frac{\tr(H^2 P^{-1})}{\tr(P^{-1})} - \frac{\tr^2(HP^{-1})}{\tr^2(P^{-1})} \geq \frac{\Vert H \Vert_F^2}{\kappa_1(P)}
\]

Since $\tr(P) = \|A\|_F^2$ and $\tr(P^{-1}) = \|A^\dag\|_F^2$, we have $\kappa_1(P) = \kappa_F^2(A)$. Thus:
\[
\langle H, \nabla^2 C_A(I)H \rangle \geq \frac{4\Vert H \Vert_F^2}{\kappa_F^2(A)}
\]

For $KX \in S$, let $B =  XA$. Then $\nabla^2 C_A(KX) = \nabla^2 C_B(I)$, and since $KX\in S$, we have $\kappa_F(B) \leq \kappa_F(A)$, so:
\[
\lambda_{\min}(\nabla^2 C_A(Kg)) \geq \frac{4}{\kappa_F^2(B)} \geq \frac{4}{\kappa_F^2(A)}
\]
\end{proof}
\begin{lemma}
\label{lem:strong-convexity-lb}
    Suppose $P\in\PD(m)$ is a positive definite matrix. 
    Then, for any $H\in \Herm(m)$ with $\tr(H)=0$ we have \begin{equation}
    \label{eq:Hessian-lb}
    \frac{\tr(H^2 P)}{\tr(P)} - \frac{\tr^2(HP)}{\tr^2(P)} 
    \geq \frac{\Vert H\Vert_F^2}{\kappa_1(P)},
    \end{equation} where $\kappa_1(P)\coloneqq \tr(P)\tr(P^{-1})$ denotes the condition number of $P$ with respect to the trace norm.
\end{lemma}
\begin{proof}
    Note that both sides of \cref{eq:Hessian-lb} are invariant under conjugating both $P$ and $H$ with a unitary matrix, and by scaling $P$ with a positive real number.
    Hence, we may assume without loss of generality that $P=\diag(p_1,p_2,\dots,p_m)$ is diagonal and $\tr(P)=1$.
    Moreover, a straightforward computation shows that \cref{eq:Hessian-lb} holds trivially for $P=\lambda I_m$.
    Hence, we may assume without loss of generality that $p_i\neq p_j$ for some $i\neq j$.

    We prove the statement in two steps.
    We first show using Weyl's inequality that \cref{eq:Hessian-lb} holds for a traceless, diagonal matrix $H$.
    Then, we prove that the left-hand side of \cref{eq:Hessian-lb} takes its minimum at a diagonal matrix, which then finishes the proof.

    We define $f(H)$ to be the left-hand side of \cref{eq:Hessian-lb}.
    We first show that $f(H)\geq \Vert H\Vert_F^2 \kappa_1^{-1}(P)$ holds for diagonal $H$ with $\tr(H)=0$. 
    Equivalently, we show that $f(H)\geq \kappa_1^{-1}(P)$ for a diagonal $H\in\Herm(m)$ with $\Vert H\Vert_F^2=1$ and $\tr(H)=0$.
    
    Assume $H=\diag(x_1,x_2,\dots,x_m)$ with $\sum_{i=1}^m x_i = 0$ and $\sum_{i=1}^m x_i^2 = 1$.
    Then, we can lower bound the term \begin{equation}
    \label{eq:var-mu}
     \frac{\tr(H^2 P)}{\tr(P)} - \frac{\tr^2(HP)}{\tr^2(P)} = x^T \left( P - pp^T \right) x
    \end{equation} by the smallest eigenvalue of $D_1 \coloneqq P - p p^T$ on the subspace $\{x\in\RR^m\mid \sum_i x_i = 0\}$ (recall that we assume $\tr P=1$). 
    In fact, note that $D_1$ is positive semi-definite, and all-ones vector $\mathbb{1}_m$ lies in its kernel.
    By $\sum_i x_i=0$, the vector $x$ is orthogonal to $\mathbb{1}_m$, so we can further lower bound \[
    x^T \left( P - p p^T \right) x = x^T D_1 x \geq \lambda_{n-1} (D_1),
    \] where $\lambda_1(D_1)\geq\lambda_2(D_1)\geq\dots\geq\lambda_{n}(D_1)$ denote the ordered eigenvalues of $D_1$ so $\lambda_{n-1}(D_1)$ is the second smallest eigenvalue of $D_1$. 
    
    Now $D_1=\diag(p)-pp^T$ is a rank one perturbation of $\diag(p)$. 
    By Weyl's inequality we have \[
    \lambda_{n-1}(D_1) \geq \lambda_{n}(\diag(p)) + \lambda_{n-1}(-pp^T) = \lambda_n(\diag(p))=p_{\min},
    \] where $p_{\min}\coloneqq \min_i p_i$. 
    Here, the second to last equality holds since $pp^T$ is rank one.
    We deduce that $f(H)\geq \lambda_{\min}(P)$ for $P$ with $\tr(P)=1$, and $f(H)\geq \lambda_{\min}(P) \, \tr^{-1}(P)$ for general positive definite $P$.  
    Hence, \[
    \frac{\tr(H^2 P)}{\tr(P)} - \frac{\tr^2(HP)}{\tr^2(P)}\geq \frac{\lambda_{\min}(P)\Vert H\Vert_F^2}{\tr(P)} \geq \frac{\Vert H\Vert_F^2}{\tr(P)\tr(P^{-1})} = \frac{\Vert H\Vert_F^2}{\kappa_1(P)} ,
    \] where we use $\frac{1}{\lambda_{\min}(P)}\leq \tr(P^{-1})$ in the last inequality.
	This proves the statement for diagonal $H$ with $\tr(H)=0$.

    We will now finish the proof by showing that the minimum of $f(H)/\Vert H\Vert_F^2$ occurs when $H$ is diagonal.
    We compute the gradient of $f$ at $H$: \[
    \nabla f(H) = \frac{HP+PH}{\tr P}-2\frac{\tr(HP)P}{\tr^2 P}. 
    \] Assume that $H$ minimizes $f(H)$ on the unit sphere $S\coloneqq \{H\in\Herm(m)\mid \Vert H\Vert_F^2=1, \tr H=0\}$.
   The tangent space $T_H(S)$ of $S$ at $H$ equals $\{H'\in\Herm(m)\mid \tr(HH')=0, \tr H'=0\}$, i.e., the orthogonal complement of the line spanned by $H$.
   Hence, we have $0=\proj_{T_H(S)}\nabla f(H)$.
   A direct computation yields $\tr\left(\nabla f(H)\right)=0$ and $\tr\left(H\nabla f(H)\right)=2f(H)$.
   This implies that $\proj_{T_H(S)}\nabla f(H) = \nabla f(H) - \tr\left(\nabla f(H)\right) I_m - \tr\left(\nabla f(H) H\right) H = \nabla f(H)-2f(H)H$ and we deduce \begin{equation}  
    \label{eq:unit-sphere-gradient}
        \nabla f(H) = 2 f(H)H.
    \end{equation}
    To reach a contradiction, we assume that $H$ is not diagonal, say $H_{ij}\neq 0$ for $i\neq j$. 
    Then, by looking at the $(i,j)$-entry of both sides of \cref{eq:unit-sphere-gradient} we have \[
    \frac{H_{ij} (p_i+ p_j)}{\sum_{i=1}^n p_i} = 2f(H) H_{ij}.
    \] Hence, if $H$ is a minimum of $f(H)$ on $S$ and $H_{ij}\neq 0$ for some $i\neq j$, then \begin{equation}
    \label{eq:f-for-non-diag}
    2 f(H) = \frac{p_i+p_j}{\sum_{k=1}^n p_k}
    \end{equation}
    Indeed, \cref{eq:f-for-non-diag} holds for $H=\frac{1}{\sqrt{2}}\left( E_{ij}+E_{ji}\right)$ where $E_{ij}$ is the elementary matrix with $1$ at $(i,j)$-th entry and zeros elsewhere.
    
    On the other hand, consider the diagonal matrix $D=\frac{1}{\sqrt{2}}\left(E_{ii}-E_{jj}\right)$ for $i\neq j$.
    Then a straightforward calculation yields \[
    2f(D) = \frac{p_i+p_j}{\sum_{k=1}^n p_k} - \left(\frac{p_i-p_j}{\sum_{i=1}^n p_i}\right)^2.
    \] Comparing with \cref{eq:f-for-non-diag}, we obtain $f(D)\leq f(H)$.
    This proves that $f$ takes its minimum on $S$ at a diagonal matrix and finishes the proof.
\end{proof}

We now prove \cref{thm:main-linear-algorithm}.
\begin{proof}[Proof of \cref{thm:main-linear-algorithm}]
We first note that $\tr\left(\nabla C_A(X)\right)=0$ for every $X\in G$. 
Hence, we have $\det\exp(-\frac{1}{L}\nabla C_A(X))=1$ by the formula $\det\exp(H)=\exp(\tr(H))$. 
Since the starting point of the gradient descent algorithm is $g_0=I_m$, this shows that the gradient descent algorithm gives iterates inside the subgroup $G\cap\SL_m(\CC)$.
Consequently, we may assume without loss of generality that $G\leq\SL_m(\CC)$.

By \cref{cor:left-preco-duality,prop:strong-convexity-parameter}:
 $C_A$ is $4$-smooth and
 $(4/\kappa_F^2(A))$-strongly convex on the sublevel set $S = \{Kg \mid C_A(Kg) \leq C_A(I)\}$.
 
 Hence, \cref{cor:strongly-convex-descent} applied to the function $C_A$ gives $\log\kappa_F(X A) - \log\kappa_\star<\varepsilon$, for 
\[
T =  \kappa_F^2(A)\,\log\left(\frac{\log(\kappa_F(A)/\kappa_F^\star)}{\varepsilon}\right)
\] iterations of the smooth gradient descent.
This concludes the proof.
\end{proof}

\section{Preconditioning Polynomial Systems} \label{sec:polysystems}
\subsection{Setup}
Suppose $d\geq 1$ is a number and let $\PP_{n,d}\coloneqq\CC[x_1,\dots,x_n]_{\leq d}$ denote the vector space of polynomials of degree at most $d$. A basis of $\PP_{n,d}$ is given by the monomials $x^\alpha,\; \alpha\in\NN^n$ where $x^\alpha\coloneqq x_1^{\alpha_1}x_2^{\alpha_2}\dots x_n^{\alpha_n}$.

The general linear group $\GL_n(\CC)$ acts on $\PP_{n,d}$ by applying a linear change of variables \[
(Y\cdot f)(x)\coloneqq f(Y^{-1}x),\quad Y\in \GL_n(\CC),\, f\in\PP_{n,d}.
\] We want to endow $\PP_{n,d}$ with an inner product that is invariant under the action of the unitary group, $\mathrm{U}_n$.
We consider the decomposition \[
\PP_{n,d} = \HH_0\oplus\HH_1\oplus\dots\oplus\HH_d
\] where $\HH_e\coloneqq\CC[x_1,x_2,\dots,x_n]_e$ is the space of homogeneous polynomials of degree $e$.  

\begin{definition}
\label{def:BW-inner}
We define the \emph{Bombieri-Weyl inner product} on $\HH_e$ as the unique inner product satisfying \[
    \langle x^\alpha, \, x^\beta\rangle=\begin{cases}
0 & \text{ if }\alpha\neq\beta\\
\frac{\alpha_1!\alpha_2!\dots \alpha_n!}{e!} & \text{ otherwise}
    \end{cases}
\] The orthonormal basis $\{\binom{e}{\alpha}^{\frac{1}{2}} x^\alpha\mid |\alpha|=e\}$ with respect to the Bombieri-Weyl inner product is called the \emph{Weyl's basis}.
In this basis we have \[
\langle f,g\rangle = \sum_{\alpha} f_\alpha \overline{g_\alpha}
\] for $f=\sum_{\alpha} \sqrt{\binom{e}{\alpha}}\, f_\alpha x^{\alpha}$ and $g=\sum_{\alpha} \sqrt{\binom{e}{\alpha}}\, g_\alpha x^{\alpha}$.
\end{definition}
The Bombieri-Weyl inner product linearly extends to $\PP_{n,d}$ by \[
\langle f,g\rangle = \sum_{k=0}^d \langle f^{(k)},g^{(k)}\rangle_{\HH_i},
\] where $f^{(k)}$ denotes the degree $k$ form of $f$. 

The following is the defining property of the Bombieri-Weyl inner product:

\begin{lemma}
\label{lem:BW-unitary-inv}
    Bombieri-Weyl inner product is unitarily invariant. 
    That is, \[
    \langle U \cdot f , U\cdot g \rangle = \langle f,g\rangle, \quad U\in\rmU_n, f,g\in\PP_{n,d}.
    \]
\end{lemma}
\begin{proof}
    See \cite[Theorem~16.3]{BC-condition-bk}.
\end{proof}

\begin{remark}
Bombieri-Weyl inner product is suited to work with dense polynomials, i.e., polynomials that have most terms with non-zero coefficients. Later we will work with sparse polynomials where we only insist on monomials being orthogonal to each other.
\end{remark}

\subsection{Polynomial Systems and Condition Number}
\label{sec:p-systems-and-condition}

Suppose $\dd=(d_1,d_2,\dots,d_m)$ is a \textit{degree pattern}. 
We define the vector space of polynomial systems with degree pattern $\dd$: 
\[ \PP_{n,\dd} \coloneqq \PP_{n,d_1}\oplus\dots\oplus\PP_{n,d_m} . \]
We note that the dimension of $\PP_{n,\dd}$ equals \[
\dim\PP_{n,\dd}= \sum_{i=1}^m \binom{n+d_i}{d_i}.
\] The Bombieri-Weyl inner product extends to systems in $\PP_{n,\dd}$ by setting \begin{equation}
\label{eq:system-weyl-inner}
	\langle \ff,\bm{g}\rangle \coloneqq \sum_{i=1}^m \langle f_i, g_i\rangle,
\end{equation}
for $\ff=(f_1,f_2,\dots,f_m)$ and $\bm{g}=(g_1,g_2,\dots,g_m)$.

\begin{definition}
\label{def:BW-norm-system}
	Let $\ff=(f_1,f_2,\dots,f_m)$ with $f_i\in\CC[x_1,\dots,x_n]$ for $i\in [m]$. 
	The Bombieri-Weyl norm of $\ff$ is defined as \[
	\Vert\ff\Vert_W \coloneqq \sqrt{\sum_{i=1}^m \langle f_i,f_i\rangle}
	\] where $\langle\cdot,\cdot\rangle$ is the Bombieri-Weyl inner product (\cref{def:BW-inner}).
\end{definition}

We consider two group actions on $\PP_{n,\dd}\times\CC^n$: 
On the one hand, the group $\GL_{n}(\CC)$ acts on $\PP_{n,\dd}$ via linear changes of variables and on $\CC^n$ by matrix-vector multiplication. 
On the other hand, $\GL_m(\CC)$ acts on $\PP_{n,\dd}$ by shuffling $f_1,f_2,\dots,f_m$, i.e., replacing each $f_i$ with a linear combination of $f_1,f_2,\dots,f_m$. 
We assume that $\GL_m(\CC)$ acts on $\CC^n$ trivially. 
It is straightforward to check that these two actions commute: For $X\in\GL_m(\CC)$ and $Y\in\GL_n(\CC)$, applying first $X$ then $Y$ to $(\ff,\xi)$ is the same as applying $Y$ first and $X$ second. 

\begin{lemma}
\label{lem:inner-invariant}
	The Bombieri-Weyl inner product on $\PP_{n,\dd}$ is unitarily invariant, that is, given two unitary matrices $U_1\in\rmU_m$ and $U_2\in\rmU_n$, \[
	\langle (U_1,U_2)\cdot \ff , (U_1,U_2)\cdot\bm{g}\rangle = \langle \ff,\bm{g}\rangle
	\] for every $\ff,\bm{g}\in \PP_{n,\dd}$.
\end{lemma}
\begin{proof}
We first prove the invariance under the action of $U_2$:\[
\langle U_2\cdot \ff,U_2\cdot \bm{g}\rangle=\sum_{i=1}^m \langle U_2\cdot f_i,U_2\cdot g_i\rangle = \sum_{i=1}^m \langle f_i,g_i\rangle = \langle \ff,\bm{g}\rangle,
\] where we use \cref{lem:BW-unitary-inv} in the second equality.

We now prove the invariance under the shuffling action of $U_1$:
\[
	\langle U_1\cdot \ff, U_1 \cdot \bm{g}\rangle = \sum_{i=1}^n \Big\langle \sum_{j=1}^m (U_1)_{ij} f_j , \sum_{k=1}^m (U_1)_{ik} g_k \Big\rangle = \sum_{i=1}^n \sum_{j=1}^m \sum_{k=1}^m (U_1)_{ij} \overline{(U_1)_{ik}}\langle f_j, g_k\rangle = \sum_{i=1}^n \langle f_i,g_i\rangle = \langle\ff,\bm{g}\rangle
	\] where the second to last equality follows from $U_1U_1^\ast = I_m$.
\end{proof}




\begin{remark}
    The action of $\GL_n(\CC)$ on polynomial systems by a change of variables does not respect sparsity patterns of the polynomial system. 
    Therefore, this action is more suitable for dense polynomials. 
    On the other hand, the shuffling action preserves the sparsity of the unmixed systems $\ff$, i.e., systems where $\mathrm{supp}(f_i)\subset\NN^{n}$ are all same.
\end{remark}

\begin{definition}
    Suppose $\Vert\cdot\Vert$ is a matrix norm on $\CC^{n\times m}$.
    Given a system $f\in\PP_{n,\dd}$ and  $\xi\in\CC^{n}$, we define the \emph{local condition number} $\mu(f,\xi)$ with respect to $\Vert\cdot\Vert$ as \[
    \mu(\ff,\xi) = \Vert \ff\Vert_W \, \Vert D^\dag_\xi(\ff)\Vert,
    \] where $D^\dag_\xi(\ff)$ denotes the Moore-Penrose inverse of the Jacobian $D_\xi(\ff)$ of $\f$ at $\xi$.
\end{definition}

We want to analyse the effect of the action of $\GL_m(\CC)\times\GL_n(\CC)$ on $\PP_{n,\dd}\times\CC^n$ on the local condition number $\mu(\ff,\xi)$. 
We first need to see how the Jacobian $D_{\xi}(\ff)$ changes under the action. 
The following follows from the chain rule.
\begin{lemma}
\label{lem:action-on-jacobian}
    Suppose $\ff\in\PP_{n,\dd}$ is a system of $m$ polynomials, $\xi\in\CC^n$. 
    For $X\in\GL_m(\CC)$ and $Y\in\GL_n(\CC)$ \[
       D_{Y\xi} ( (X,Y)\cdot f ) = X D_{\xi}(f) Y^{-1}.
    \] 
\end{lemma}

For our purposes, the following observation is crucial. 
\begin{corollary}
	Suppose $\Vert\cdot\Vert$ is a unitarily invariant matrix norm on $\CC^{n\times m}$.
    Then the local condition number is unitarily invariant. 
    That is, for every $(\ff,\xi)\in\PP_{n,\dd}\times\CC^{n}$ and two unitary matrices $U_1\in \mathrm{U}_{m}$ and $U_2\in \mathrm{U}_n$ we have \[
    \mu( (U_1,U_2)\cdot \ff, U_2\xi ) = \mu(\ff,\xi).
    \] 
\end{corollary}
\begin{proof}
	By \cref{lem:action-on-jacobian,lem:inner-invariant} we have \[
	\mu\left( (U_1,U_2)\cdot \ff, U_2\xi \right) = \Vert (U_1,U_2)\cdot \ff\Vert_W\, \Vert U_2 \, D^\dag_{\xi}(\ff) U_1^{-1}\Vert = \Vert \ff\Vert_W \, \Vert D^\dag_{\xi}(\ff) \Vert. 
	\] 
\end{proof}


We now prove \crefpart{thm:main-g-convexity}{thm:main-part-2}.

\begin{proof}[Proof of \crefpart{thm:main-g-convexity}{thm:main-part-2}]
    We can write $C_{(\ff,\xi)}(X,Y) = \log \Vert (X,Y)\cdot \ff\Vert_W + \log \Vert YD_{\xi}(\ff)X^{-1}\Vert_2$.
    The first summand is geodesically convex by \cref{ex:Kempf-Ness}, since it is the Kempf-Ness function associated to $\ff$ and $\Vert\cdot\Vert_W$ arises from the unitarily invariant inner product (\cref{lem:inner-invariant}). 
    The geodesic convexity of the second summand follows from \cref{prop:convexity-of-log-norm}. 
    Hence, $C_{(\ff,\xi)}$ the sum of two geodesically convex functions, so it is geodesically convex.
\end{proof}

\subsection{Algorithm for the shuffling action}
\label{sec:shuffling-algo}
In this section we prove \cref{thm:shuffling-iteration}.
As we discussed in \cref{sec:intro-preco-non-linear}, the main point is that there is a reduction from preconditioning non-linear systems under the shuffling action to the optimization of the cross condition number, $\kappa_F(A,B)=\Vert A\Vert_F\Vert B\Vert_F$.

\begin{lemma}
\label{lem:S_f}
	Let $\ff$ be a polynomial system and $\xi\in\CC^n$. 
	Then there exists a matrix $S_{\ff}$ such that \[
	\mu_F(X\cdot\ff,\xi) = \Vert X S_{\ff}\Vert_F \, \Vert D^\dag_{\xi}(\ff) X^{-1}\Vert_F
	\] for every $X\in\GL_m(\CC)$.
\end{lemma}

Before we prove it, we use \cref{lem:S_f} to prove \cref{thm:shuffling-iteration}.

\begin{proof}[ Proof of \cref{thm:shuffling-iteration} ]
Denote $A\coloneqq S_{\ff}$ and $B\coloneqq D^{\dag}_{\xi}(\ff)$.
	By \cref{lem:S_f} we have $
	\mu_F(X\cdot\ff,\xi) = \Vert XA\Vert\,\Vert BX^{-1}\Vert_F
	$, which implies that $C_{(\ff,\xi)}(KX) = C_{A,B}(KX)$ for every $X\in\GL_m(\CC)$.
	The theorem then follows from \cref{thm:cross-condition-iteration}.
\end{proof}

We now prove \cref{lem:S_f}:

\begin{proof}[Proof of \cref{lem:S_f}]
	Let $G_{\ff}\in\CC^{m\times m}$ be the Gram matrix of $\ff$, i.e., \[
	(G_{\ff})_{ij} = \langle f_i, f_j\rangle,\quad 1\leq i,j\leq m.
	\] Then, $G_{\ff}$ is Hermitian since the Bombieri-Weyl inner product is Hermitian, and it is positive semi-definite since \[
	v^{\ast} G_{\ff} v = \Big\|\, \sum_{i=1}^m \overline{v_i} f_i\,\Big\|_W \geq 0.
	\] A straightforward computation shows that $G_{X\cdot\ff}=X G_{\ff}X^\ast$ for $X\in\GL_m(\CC)$.
	
	Since $G_{\ff}$ is positive semi-definite, it has a Hermitian, positive semi-definite square root $S_{\ff}=\sqrt{G_{\ff}}$.
	By $G_{X\cdot \ff}=X G_{\ff}X^{\ast}$, we have $S_{X\cdot\ff} = X S_{\ff}$.
	Moreover, \[
	\Vert \ff\Vert_W^2 = \sum_{i=1}^m \langle f_i,f_i\rangle = \tr(G_{\ff}) = \Vert S_{\ff}\Vert_F^2.
	\] Overall, we deduce that \[
	\Vert X\cdot\ff\Vert_W = \Vert XS_{\ff}\Vert_F
	\] and this finishes the proof.
\end{proof}

\subsection{Shuffling and Change of Variables}
\label{sec:general-poly-algo}

	



We now consider the general case and assume that $G\leq\GL_m(\CC)\times\GL_n(\CC)$ is a symmetric subgroup with $T_m(\CC)\times T_n(\CC)\leq G$. 

We start by computing the weights of the action of $\GL_m(\CC)\times\GL_n(\CC)$ on $\PP_{n,\dd}\otimes \CC^{n\times m}$ via \begin{equation}
\label{eq:system-tensor-action}
(X,Y)\cdot (\ff \otimes A) \coloneqq \left( (X,Y)\cdot\ff \right) \otimes YAX^{-1}.
\end{equation}

\begin{lemma}
	The weights of \eqref{eq:system-tensor-action} are of the form \[
	\omega_{i,j,\alpha,l} \coloneqq ( e_i -e_j, -\alpha + e_l ) \in\ZZ^{m+n},
	\] where $1\leq i,j\leq m$, $1\leq l\leq n$ and $\alpha\in\NN^n$ satisfies $|\alpha|\leq D$.
\end{lemma}
\begin{proof}
    $\PP_{n,\dd}$ admits a basis of the form \[
    \mathfrak{m}_{i,\alpha} = (0,\dots,0, x^\alpha,0,\dots,0)
    \] where $|\alpha|\leq d_i\leq D$. 
    Then $\mathfrak{m}_{i,\alpha}\otimes E_{kj}$ is a basis of $\PP_{n,\dd}\otimes\CC^{n\times m}$, where $E_{kj}\in\CC^{n\times m}$ is the elementary matrix with a $1$ in the $kj$-entry and zeros elsewhere.
    
    For a pair of diagonal matrices $X=\diag(t_1,\dots,t_m), Y=\diag(s_1,\dots,s_n)$ we compute \[
    (X,Y)\cdot (\mathfrak{m}_{i,\alpha}\otimes E_{kj}) = ( t_i s^{-\alpha} \mathfrak{m}_{i,\alpha} ) \otimes (t_j^{-1} s_k E_{kj}) = t_i t_j^{-1} s^{-\alpha}s_k \mathfrak{m}_{i,\alpha}\otimes E_{kj}.
    \] This shows that $\mathfrak{m}_{i,\alpha}\otimes E_{kj}$ is a weight vector of weight $\omega_{i,j,\alpha,l}$.
    Since $\mathfrak{m}_{i,\alpha}\otimes E_{kj}$ form a basis, all weights are of this form.
\end{proof}

\begin{lemma}
\label{lem:change-norm-margin}
	The weight norm and the weight margin of \eqref{eq:system-tensor-action} satisfy \[
	N \leq D+2, \quad \gamma\geq (D+2)^{1-m-n} (m+n)^{-1}. 
	\]
\end{lemma}
\begin{proof}
	We compute \[
	\Vert \omega_{i,j,\alpha,l} \Vert^2 = \Vert e_i-e_j\Vert^2 + \Vert \alpha+e_l\Vert^2 \leq 2 + \Vert \alpha+e_l\Vert_1^2 \leq 2 + (D+1)^2 \leq (D+2)^2.  
	\] This proves $N\leq D+2$. 
	
	For the lower bound of the weight margin, we apply \cref{prop:margin-general-lb}: $\gamma\geq N^{1-m-n}(m+n)^{-1}\geq (D+2)^{1-m-n} (m+n)^{-1}$.
\end{proof}

We can now prove \cref{thm:change-iteration}.
\begin{proof}[Proof of \cref{thm:change-iteration}]
The proof is analogous to the proof of \cref{thm:cross-condition-iteration}.

	By \cref{lem:change-norm-margin}, the local condition number function $C_{(\ff,\xi)}$ is $L\coloneqq (D+2)$-smooth.
	Set $\varepsilon'\coloneqq \varepsilon\gamma$ where $\gamma\coloneqq (D+2)^{1-m-n}(m+n)^{-1}$.
	
	By \cref{prop:grad-eps}, the gradient descent with step-size $\eta=L^{-1}$ and $(X_0,Y_0)=(I_m,I_n)$ reaches a group element $(X_T,Y_T)$ with $\Vert\nabla C_{(\ff,\xi)}\Vert<\varepsilon'$ in \[
	T = O\left( \frac{L}{(\varepsilon')^2}\,\log\left(\frac{\mu_F(\ff,\xi)}{\mu_F^\star}\right)\, \right) = O\left( \frac{L}{\gamma^2 \varepsilon^2} \log\left(\frac{\mu_F(\ff,\xi)}{\mu_F^\star}\right)\right)
	\] iterations. 
	An application of \cref{thm:nc-duality} using the weight margin bound \cref{lem:change-norm-margin} yields \[
	\left(\frac{\mu_F^\star}{\mu_F((X,Y)\cdot\ff,Y\xi)}\right)^2 = \left(\frac{\min_{(X',Y')\in G} \Vert (X',Y')\cdot (\ff\otimes D^\dag_{\xi}(\ff))\Vert}{\Vert (X,Y)\cdot (\ff\otimes D^\dag_{\xi}(\ff))\Vert}\right)^2 \geq 1-\frac{\varepsilon'}{\gamma}=1-\varepsilon.
	\] Taking logarithms and using $\log(1-\varepsilon)<2\varepsilon$ for $\varepsilon<0.5$ gives \[
	2 \log\left( \frac{\mu_F((X,Y)\cdot\ff,Y\xi)}{\mu_F^\star} \right) \leq -\log(1-\varepsilon)<2\varepsilon,
	\] and this finishes the proof.
\end{proof}

\subsection{Preconditioning Sparse Polynomial Systems} \label{sparse}
This section concerns preconditioning of unmixed sparse polynomial systems. 
In this setting, it is natural to preserve the sparsity patterns to keep advantage of lower evaluation complexity compared to a dense polynomial equation. 
The group that preserves sparsity is $\GL_m(\CC)\times(\mathbb{C}^{\times})^n$, where $\GL_m(\CC)$ acts by the shuffling action and $(\CC^\times)^n$ acts by scaling the variables. 

We again consider the Weyl inner product on the space of sparse polynomial systems and we aim to do optimization on \[
\left(\GL_m(\CC)\times (\CC^\times)^n\right) / (\mathrm{U}_m(\CC)\times (S^1)^n) \cong \PD_m(\CC) \times \RR_{>0}^n.
\]

There are two objectives in preconditioning of sparse polynomial systems. 
Let $\f= (f_1,\ldots,f_n)$ be the system of equations and $\xi \in (\mathbb{C}^{\times})^n$ be a root of $\f$.
\begin{enumerate}
    \item First objective is to scale the variables to ensure that the proportions of magnitudes $\frac{|\xi_i|}{|\xi_j|}$ of the root are close to $1$.
    \item Second objective is to use the action of $\GL_m(\CC)$ to reduce the local condition number $\mu(\f,\xi)$. 
\end{enumerate}
It is also possible to use several roots on $(\mathbb{C}^{\times})^n$ and precondition them all together. 
We will formulate a scheme that generalizes to several points seamlessly.

 To our knowledge, the earlier works on this problem are \cite{verschelde2000toric} and in \cite{malajovich2023complexity}. 
 In \cite{verschelde2000toric}, Verschelde uses Cox-coordinates for representing the solution space and aims to use homogeneities of the Newton polytope to scale coordinates of the root vector. 
 The main in this paper is to make the coordinates of the root vector comparable to each other; essentially addressing the first intuition above.  
 The paper does not provide any theoretical analysis on the impact of the proposed idea to condition numbers.  
 In \cite{malajovich2023complexity}, Malajovich uses a group action that he names "renormalization" to bring every root to $(1,1,1,\ldots,1)$, do the homotopy step there, and then revert back. 
 The paper build a rigorous theory to analyze complexity of homotopy algorithm when renormalization is applied at every step. 
 Results of Malajovich are essentially the first rigorous work to analyze complexity of solving sparse polynomial system with homotopy continuation, and this is the main purpose of the paper. 
 Our approach is much more simple and directed towards preconditioning, and it has the advantage of using two group actions at the same time.

Now we introduce our approach. 
The action of $t=(t_1,t_2,\ldots,t_n) \in \mathbb{R}_{>0}^n$ on a polynomial $p$ is defined as 
\[ (t \cdot p) (x_1,x_2, \cdots,x_n) \coloneqq p(t_1^{-1} x_1, t_2^{-1} x_2, \ldots, t_n^{-1} x_n), \]
where the inverses make $t\cdot p$ a left action.

We now introduce a function $h_{\xi}(t)$ on $\RR_{>0}^n$, which is geodesically convex and takes its minimum at a point $t$ which satisfies $\forall i,j, |\xi_i|t_i^{-1}=|\xi_j|t_j^{-1}$.
To express $h_{\xi}(t)$, we define the exponent vectors $\omega_i\coloneqq n e_i -\mathbb{1}_n, i=1,2,\dots,n$, where $e_i$ denotes the $i$-th standard vector and $\mathbb{1}_n$ denotes the all-ones vector. 
We set
\[ h_{\xi}(t) \coloneqq \log\left(\, \sum_{i=1}^n  |\xi|^{\omega_i} \, t^{-\omega_i} \,\right), \]
defined on $\RR_{>0}^n$ where $|\xi|^{\omega_i} \coloneqq \prod_{i=1}^n |\xi_j|^{\omega_{ij}}$.

\begin{lemma}
   Suppose $\xi \in (\CC^\times)^n$. 
   Then $h_{\xi}(t)$ is a geodesically convex function on $(\CC^\times)^n / (S^1)^n\cong \mathbb{R}_{>0}^n$.
    Moreover, every minimizer $t$ of  $h_{\xi}$ satisfies \[
    \forall 1\leq i,j\leq n,\qquad |\xi_i| t_i^{-1} = |\xi_j| t_j^{-1}.
    \]
\end{lemma}
\begin{proof}
   By \cref{def:geodesic-convexity}, $h_{\xi}(t)$ is geodesically convex if and only if \[
   h_{\xi}(e^{sv}t) = \log\left(\; \sum_{i=1}^n |\xi|^{\omega_i}\, t^{-\omega_i}\, e^{s \langle v,\omega_i\rangle} \; \right)
   \] is convex as a function of $s\in\RR$ for every 
   $v\in\RR^n$ and $t\in (\mathbb{R}_{>0})^n$. 
   Note that each summand $|\xi|^{\omega_i}\, t^{-\omega_i}\, e^{s \langle v,\omega_i\rangle}$ is a logarithmically convex function of $s$ and a sum of logarithmically convex functions is logarithmically convex.
   This implies that $h_{\xi}(e^{sv}t)$ is a convex function of $s$.

   We now prove the second claim. 
   A direct computation yields \[
   \begin{split}
       \langle v,\nabla h_{\xi}(t)\rangle &=\partial_{s=0}\; h_{\xi}(e^{sv}t) \\
       &=\frac{\sum_{i=1}^n |\xi|^{\omega_i}t^{-\omega_i} \langle v,\omega_i\rangle }{\sum_{i=1}^n |\xi|^{\omega_i}t^{-\omega_i}}.
   \end{split}
   \] Hence, $\nabla h_{\xi}(t)=0$ if and only if $\sum_{i=1}^n |\xi|^{\omega_i} t^{-\omega_i}\omega_i=0$.
   By the definition of $\omega_i$,  we have \[
  \nabla h_{\xi}(t) = 0\; \implies \;\forall i,j,\; |\xi|^{\omega_i} t^{-\omega_i} = |\xi|^{\omega_j} t^{-\omega_j}. 
   \] This implies $|\xi_i| t_i^{-1}=|\xi_j| t_j^{-1}$ and finishes the proof.
\end{proof}

We design our optimization problem as follows:

 \begin{equation} \label{sparse-precond}
     \min_{X \in \PD(n), t \in \mathbb{R_{+}}^n} \mu\left(X \circ ( t \cdot  p) ,(t^{-1} \cdot \xi) \right) + h_{\xi}(t)  
 \end{equation} 
where $\mu$ is the local condition number. Note that $h_{\xi}(t)$ is independent of the matrix $X$.
\begin{proposition}
The optimization problem posed in 
    \Cref{sparse-precond} is a geodesically convex optimization problem on the group $\PD(n) \times \RR_{>0}^n$.
\end{proposition}
\begin{proof}
$\mu\left(X \circ ( t \cdot  p) ,(t^{-1} \cdot \xi) \right)$ is geodesically convex on the entire space as follows from \Cref{thm:main-g-convexity}. $h_{\xi}(t)$ is geodesically convex also, and the sum of two geodesically convex functions is geodesically convex.
\end{proof}

\subsection{Applications}

\subsubsection{Preconditioning subdivision algorithms}
Subdivision methods are common in computational geometry and computational algebra over the reals. 
Suppose we are given a system of equations $\ff = (f_1,\ldots,f_m)$ with common real zero set $\mathrm{Z}(\ff)$. 
The basic idea of these methods is to identify conditions on a hypercube $B$ that guarantees either $B \cap \mathrm{Z}(f) = \emptyset$
or $B \cap \mathrm{Z}(f)$ is isomorphic to a linear object.  
A special case of the subdivision method involves root isolation (when $Z(\ff)$ is finite): 
The core idea is to repeatedly subdivide $B$ into smaller boxes until each box contains at most one root of $Z(\ff)$.

Existing subdivision algorithms \cite{subdivision-yap, AME-11, subdivison-mourrain} share a common feature:
Their complexities depend on the local condition number, typically evaluated at the center of the box.
For instance, the quantity $\lambda_1(\alpha)$ which makes an appearance in the complexity analysis of the subdivision method in \cite{subdivision-yap} implicitly encodes a condition number, and the complexity of the algorithm in  \cite{AME-11} is already analyzed using the local condition number, see Thm 6.5 \cite{AME-11}. 
In \cite{cucker_complexity_2022} this connection is clarified and much improved estimates are obtained on the complexity of a popular subdivision algorithm due Plantinga and Vegter. 
This suggests that preconditioning could significantly improve the performance of any subdivision algorithm.

The importance of preconditioning has already been noted in the aforementioned papers. 
For example, in \cite{subdivision-yap,subdivison-mourrain}, the authors consider a \textit{local preconditioning} step that we will call the \textit{inverse Jacobian scaling} (IJS, for short). 
That is, provided that the system $\ff=(f_1,f_2,\dots,f_k)$ is \textit{square} ($k=n$), and the Jacobian $D_\xi(\ff)\in\RR^{n\times n}$ is invertible, one replaces the original system $\ff$ with the system $\tilde{\ff}=(D_\xi(\ff))^{-1}\, \ff$. 
What essentially this method achieves is that the new system satisfies $D_\xi(\tilde{\ff})=I_n$, i.e., for every $i$, the tangent space to the zero locus of the each individual polynomial $f_i$ at $\xi$ is orthogonal to the axis line in the $i$-th direction. In terms of the local condition number, however, there is no guarantee that IJS decreases the local condition number. 

Our algorithm \cref{thm:shuffling-iteration} addresses this gap by provably improving the condition number, thereby reducing the complexity, and by preserving the geometric structure: 
The shuffling action leaves the ideal $I(\ff)$ and the zero locus $Z(\ff)$ unchanged.

\subsubsection{Preconditoning Homotopy Continuation Algorithms}
Homotopy continuation algorithms are widely used to solve polynomial systems over complex numbers \cite{Li-1997, Chen-Li-2015, burgisser2011problem, BC-condition-bk, Breiding-Timme-18}.
Two common issues in these algorithms is (1) ill-conditioning of a root, (2) well-conditioned but badly scaled roots. The latter looks numerically equivalent to having a root at "infinity".  
Every practitioner in the field seems to have a bag of tricks to deal with these issues. However, to the best of our knowledge, all these tricks are heuristics such as "scale  all coefficients using the largest one" or "normalize all polynomials to have similar norms". The only detailed work on the subject seems to be due to Morgan around 80's according to his book which is republished by SIAM \cite{morgan2009solving}.  Morgan gives two algorithms that implement the following two intuitions:
\begin{enumerate}
    \item Norm of the coefficient vector of each polynomial in the polynomial system should be similar.
    \item Coordinates of a root vector that is being tracked in the homotopy continuation should be as balanced as possible.
\end{enumerate}
These two ideas are intuitively clear, however, Morgan does not provide any justification of the method in terms of condition number analysis or mathematically otherwise.  We believe Morgan's ideas are useful, and this analysis remains an interesting research problem. We take a different route: We design preconditioning schemes that are guaranteed to make the condition number better. 

\begin{enumerate}
    \item Our first algorithm \Cref{precond-homotopy-1} follows a similar intuition to the first idea above: We use an $\GL_n(\CC)$ action on the polynomial equations to reduce the condition number. This action basically shuffles the equations and does not change the zero set at all.  We minimize the condition number under this action as it is a geodesically convex optimization problem.
    \item Our second algorithm  \Cref{precond-homotopy-2} implements both ideas above at the same time: It acts on the system of equations with  $\GL_n(\CC)$ and on the variables with  $\GL_n(\CC)$ at the same time. We call this double-preconditioning, and since it acts on the variables, it does change the geometry of the zero set. The change, however, is linear and easy to invert if necessary. It is a geodesically convex optimization problem to minimize the condition number under this $\GL_n(\CC)  \times  \GL_n(\CC)$ action. 
\end{enumerate}

\begin{algorithm} 
\caption{Preconditioned Linear Homotopy}
\begin{algorithmic} \label{precond-homotopy-1}
   \State \textbf{Input}: $\ff , \bm{g} \in \PP_{n,\dd}$, $\xi \in \CC^n$, $\ff(\xi)=0$.
   \Require $g(\xi)=0$
    \State $t=0$ , $q=g$, $x=\xi$
    \Repeat 
    \State Find $X \in \GL_n(\CC)$ such that $\mu_F(X\cdot q, \xi)$ is minimized.
    \begin{equation*}
    \begin{split}
    \Delta t &=  \frac{c}{d(X \cdot f, X \cdot q) \mu_F^2(X \cdot q, x)} \\
    t &= \min \{ t + \Delta t , 1 \} \\
     \tilde{q} &= t (X \cdot \ff) + (1-t) (Q \circ q) \\
     \tilde{x} &= N_{\tilde{q}}(\tilde{x}) \\
     q &= X^{-1}\cdot\tilde{q} \\
    \end{split}
    \end{equation*}
    \Until{$t=1$}
    \State \Return $x \in \mathbb{C}_{*}^n$ and halt.
\end{algorithmic}
\end{algorithm}

\begin{algorithm} 
\caption{Double-Preconditioned Linear Homotopy}
\begin{algorithmic} \label{precond-homotopy-2}
   \State \textbf{Input}: $\ff=(f_1,f_2,\ldots,f_n) , \bm{g}=(g_1,\ldots,g_n) \in \PP_{n,\dd}$, $\xi \in \mathbb{C}^n$, $\ff(\xi)=0$.
   \Require $g(\xi)=0$
    \State $t=0$ , $q=g$, $x=\xi$
    \Repeat 
    \State Find $(X,Y) \in \GL_n(\CC)\times \GL_n(\CC)$ such that 
   $\mu_F((X,Y)\cdot\ff,Y\xi)$ is minimized.
    \begin{equation*}
    \begin{split}
    \Delta t &=  \frac{c}{d( (X,Y)\cdot\ff, (X,Y)\cdot q ) \mu_F^2( (X,Y)\cdot\ff, Y x)} \\
    t &= \min \{ t + \Delta t , 1 \} \\
     \tilde{q} &= t (X,Y)\cdot\ff + (1-t) (X,Y)\cdot q \\
     \tilde{x} &= N_{\tilde{q}}(\tilde{x}) \\
     q &= Q^{-1} \tilde{q} \\
     x &= Q \tilde{x}
    \end{split}
    \end{equation*}
    \Until{$t=1$}
    \State \Return $x \in \mathbb{C}_{*}^n$ and halt.
\end{algorithmic}
\end{algorithm}

These algorithms are intended to be used in dense setting where polynomials have most coefficients non-zero. We wrote them in the context of linear homotopy; they can be easily adapted to other homotopy algorithms. To the best of our knowledge, the two algorithms above are the first rigorous methods in literature for preconditioning dense polynomial systems.

Probabilistic analysis of the linear homotopy algorithm is subject to a relatively mature literature under the umbrella of Smale's 17th problem \cite{burgisser2011problem}. The main idea in this literature is to use smoothed analysis setting and argue that ill-conditioning is a low-probability event. This very valuable as a theoretical result. In practical use of homotopy algorithms, however, the practitioner has a particular collection of equations coming from an engineering or basic science problem. These particular equations often can be ill-conditioned and probabilistic analysis does not give much insight on how to remedy this situation. Our algorithms here are intended to be used in the cases where one faces slow performance in homotopy continuation due to ill-conditioning.

\subsubsection{Preconditioning for Real Feasibility Testing In Quadratic Systems}
Given a system of quadratic equations, is there a way to quickly tell if they have a common real zero?  In at most generality, one does not expect a positive answer to this question. However, Barvinok and Rudelson were able to prove a sufficient condition for existence of a common real zero \cite{barvinok2022system} stated below.
\begin{theorem}[Barvinok, Rudelson] \label{quad-feas}
   There is an absolute constant $c >0$ such that the following holds. Let $Q_1,Q_2,\cdots,Q_m$, $ m\geq 3$, be real symmetric matrices such that 
   \[  \mathrm{Trace}(Q_i) = 0 \; \text{for} \; i =1,2,3,\cdots,m  \]
   and let $q_i: \mathbb{R}^n \rightarrow \mathbb{R}$,
   \begin{equation} \label{quadratic}
   q_i(x) = \langle Q_i x , x \rangle  \; \text{for} \; i =1,2,3,\cdots,m 
   \end{equation}
   be the corresponding quadratic forms. Let $A_1,A_2, \ldots, A_m$ be an orthonormal basis for the linear space spanned by $Q_1,Q_2,\ldots,Q_m$. Then the system \eqref{quadratic} has a real solution $x \neq 0$ if 
   \begin{equation}
       \label{eq:criterion}
  \norm{\sum_{i=1}^m A_i^2}_F \leq \frac{c}{m}. \end{equation}
\end{theorem}

We note the obvious symmetry of the quadratic system:
For any $Y \in \GL_n(\RR)$, the system 
\begin{equation} \label{act-quad}
q_i(x) = \langle Q_i Y x , Y x \rangle  \; \text{for} \; i =1,2,3,\cdots,m
\end{equation}
has a real root if and only if the system \eqref{quadratic} has a real root. 
In other words, the existence of a root of \eqref{quadratic} does not change if we replace $Q_i$ by $Y^T Q_i Y$ for some $Y\in\GL_n(\RR)$.

However, the criterion given in the theorem is not invariant under this action.
To keep the symmetry of matrices we will use an action of $\PD(n)$ on the system \eqref{quadratic}. If we use the action directly as defined in \eqref{act-quad} the criterion in \Cref{quad-feas} becomes
\[ \norm{\sum_{i=1}^m Y^{T}A_iY Y^T A_i Y} \leq \frac{c}{m} , \]
and we don't have a quick way to minimize this complicated sum. We design a relaxation of this optimization problem below. We note that 
\[ \mathbb{E}[ \big( \sum_{i=1}^m \xi_i A_i \big)^2 ] = \sum_{i=1}^m A_i^2 ,  \]
for independent Rademacher random variables $\xi_i \in  \{-1, +1 \}$.  So, we do the following: we sample a realization of $\xi_i$ for $i=1,2,\ldots,m$ say $\psi_i$ and define 
\[  B_{\psi} := \psi_1 A_1 + \psi_2 A_2 + \ldots + \psi_n A_n 
\enspace .\]
Here is our optimization problem:
\begin{equation} \label{relax}
    \min_{Y \in \PD(n)} \norm{Y^T B_{\psi} Y}
\end{equation}
Note that this is a geodesically convex optimization problem, and we have
\[  \norm{A_1^2 + A_2^2 + \cdots + A_m^2}  = \mathbb{E}_{\psi} [\norm{ (Y^T B_{\psi} Y)^2 }] \leq \mathbb{E}_{\psi} \norm{Y^T B_{\psi} Y}^2 \enspace.
\]

\section*{Acknowledgements}
L.D. is supported by the European Union (ERC Grant SYMOPTIC, 101040907) and by the Deutsche Forschungsgemeinschaft (DFG, German Research Foundation, 556164098). A.E. is supported by NSF-CCF-2414160.

\bibliographystyle{amsalpha}
\bibliography{references}

\appendix

\include{full-experiments.tex}

\end{document}

%% file: elaborate-intro.tex
\section{Introduction}
\label{sec:elaborate-intro}
Efficiently and accurately solving  large systems 
of linear and polynomial equations
is one of the most fundamental computational problems with 
applications  ranging from physics and engineering to machine learning and optimization.
In many cases of practical interest, the linear (and polynomial) equations are \emph{ill-conditioned}, 
that is, they are highly sensitive to numerical errors. 
Small perturbations, whether coming from rounding operations or measurement noise, can lead to significant inaccuracies of the computed solutions. 
This sensitivity also affects iterative solvers; for ill-conditioned systems, their behavior becomes unstable or even unpredictable, resulting in inaccurate solutions, stagnation, or prohibitively slow convergence.

A commonly accepted way to measure such numerical behavior is to use a quantity called \emph{condition number}. 
Intuitively, the condition number measures how close a given instance of a problem is to the set of ill-posed problems; the latter contains instances that are nearly impossible to solve in a numerically accurate way.  More precisely, the condition number is proportional to the reciprocal of the distance between the given problem instance and the locus of ill-posed problems; a small condition number indicates a well-conditioned problem and a large condition number indicates proximity to the ill-posed locus ~\cite{BC-condition-bk}.

To mitigate ill-conditioning, the computational science community developed \emph{preconditioning}.   The core idea is to apply algebraic transformations to the original linear (or non-linear) system of equations to obtain a new system with a much smaller condition number. 
The trade-off in preconditioning is that the algebraic transformation needs to be "undone" after the solution is obtained; therefore, the preconditioning operation should be computationally cheap to invert. 
Successful preconditioning leads to faster convergence of iterative solvers
and better quality of computed solutions. 
In modern large-scale numerical computations, preconditioning is indispensable; its effectiveness often determines whether a problem is computationally tractable, particularly in high-dimension and real-world applications, for examples see ~\cite{Saad-ibook-03}.

\medskip

In this introductory section we provide an exposition of preconditioning, covering both classical results and recent advances. 
We introduce key norms and condition numbers and formulate the optimization problems underlying optimal preconditioning. 
Our focus extends beyond diagonal preconditioners by leveraging geodesic convexity---a framework that enables theoretical guarantees for broader classes of preconditioners in both linear and nonlinear settings.
As a first motivating example, we consider solving the linear system \begin{equation}
\label{eq:linear-system}
  Ax = b  \enspace ,
\end{equation}
 where $A\in\CC^{m\times n}$ and $b\in\CC^m$. 
The first order iterative methods for solving linear systems, as in \eqref{eq:linear-system}, such as Jacobi and Gauss-Seidel \cite[Chapter~11]{GolVLo-matrix-book-13}, \cite[Chapter~4]{Saad-ibook-03}, conjugate gradient~\cite[Chapter~6]{Saad-ibook-03}, and Krylov subspace method~\cite[Chapter~6 \& 7]{Saad-ibook-03}, 
exhibit convergence rates that scale with the \emph{condition number} of $A$; this is
\begin{equation}
    \label{eq:condition-number-linear}
    \kappa(A) \coloneqq \Vert A\Vert \, \Vert A^{\dagger}\Vert,
\end{equation}
where $\Vert A\Vert$ denotes the operator norm of $A$ and $A^{\dagger}$ is the Moore-Penrose pseudoinverse of $A$.
Different (matrix) norms lead to different condition numbers. 
To distinguish the condition number in \eqref{eq:condition-number-linear}
from the others, we call it \emph{Euclidean condition number}. 

The preconditioning consists in multiplying $A$ with matrices $X$ and $Y$, so that the resulting matrix%
\footnote{We use $Y^{-1}$ instead of $Y$ to make $(X,Y)\cdot A\coloneqq XAY^{-1}$ a group action; this will be useful in the sequel.}
 $A'\coloneqq XAY^{-1}$ has a smaller condition number than $A$,
which in turn leads to improved numerical accuracy.
Moreover, our goal is to invert $A'$ faster than $A$. 
After we solve the \text{preconditioned} system $A'\tilde{x}=Xb$,
we can recover the solution of the original as $x=Y^{-1}\tilde{x}$.

In general, a good preconditioner must satisfy two key requirements. 
First, it should be computationally efficient to construct and second, solving the preconditioned system should be significantly easier than solving the original system.  
However, these requirements are often in competition with each other.
A common strategy to balance these competing demands is to impose a structure on the preconditioner. 
Popular choices include diagonal preconditioning, block diagonal preconditioning, and several others, e.g.,~\cite[Chapters~9 \& 10]{Saad-ibook-03}.

There exist several heuristic preconditioners which are widely used in practice, despite their lack of theoretical guarantees. 
Some of the most common are the following:
\begin{itemize}[leftmargin=20pt]
    \item \emph{Jacobi preconditioner:} Choose diagonal matrices $X$ and $Y$ such that $\diag(XAY^{-1})=(1,1,\dots,1)$ \cite[Chapter~11]{GolVLo-matrix-book-13}, \cite[Chapter~4]{Saad-ibook-03}. 
    For example $D^{-1}A$ or $D^{-\frac{1}{2}}AD^{-\frac{1}{2}}$ when $A$ is symmetric, where $D=\diag(A)$ is the diagonal of $A$. 
    \item \emph{Matrix equilibration:} Choose diagonal matrices $X$ and $Y$ such that $XAY^{-1}$ has equal row and column sums. 
    The Sinkhorn-Knopp algorithm \cite{Sinkhorn-64} computes $X$ and $Y$ in polynomial time.
    \item \emph{Row/column balancing:} Choose a diagonal matrix $X$ (resp. $Y$) such that the rows of $XA$ (resp. the columns of $AY^{-1}$) have the same $\ell_p$ norm, e.g.,~\cite{ORY-lp-balance-17}. 
     This is closely related to the Jacobi preconditioner: $XA$ has balanced rows in $l_2$ norm iff the diagonal entries of $XAA^\ast X^\ast$ are all the same. 
    So, row balancing $A$ wrt the $l_2$ norm is equivalent to the Jacobi preconditioning of $AA^\ast$.  
\end{itemize}
Although we frequently employ these preconditioners in practice, they do not guarantee a reduction in the condition number and, in some cases, may even increase it.
\begin{example}
Consider the matrix \[
    A=\begin{pmatrix} 3 & 0 & 0\\ 1 & 1 & 0\\ 0 & 3 & 1\end{pmatrix}
    \quad \text{ with } \quad \kappa(A)\approx 11.77,
    \]
    where $\kappa(A)$ is the Euclidean condition number. 
	The Jacobi preconditioner for $A$ could be either $X=\diag(A)^{-1}$ or $X=\diag(A)^{-1/2}$ and $Y=\diag(A)^{1/2}$.
	We have 
	\[
    XA = \begin{pmatrix}
        1 & 0 & 0\\ 1 & 1 & 0\\ 0 & 3 & 1
    \end{pmatrix}
    \text{ with } \kappa(XA)\approx 15.35 
    \ \text{ and } \ 
    XAY^{-1} = \begin{pmatrix}
        1 & 0 & 0\\ \frac{1}{\sqrt{3}} & 1 & 0\\ 0 & 3 & 1
    \end{pmatrix}
    \text{ with } \kappa(XAY^{-1})\approx 12.59 .
    \]
    Hence, the Jacobi preconditioner increases the condition number of $A$ in both cases. 
    In fact, for a Gaussian matrix $A$ (entries drawn independently and identically from the normal distribution), experiments suggest that both left and left-right Jacobi preconditioner increases $\kappa(A)$ with high probability. 
 \end{example}

\subsection*{Optimal diagonal preconditioning for linear systems}
The quest for \emph{optimal diagonal preconditioner}, i.e., diagonal matrices that maximally reduce the condition number \cite{Forsythe-Straus-55, Bauer-63, Bauer-69, vdSluis-69},
aimed to provide theoretical guaranties to heuristic preconditioning techniques. The quest was successful; \emph{when} the condition number relies on the $\ell_1$ and $\ell_\infty$ norms. 
In particular, if we consider the \emph{maximum absolute row sum norm} and the corresponding condition number \[
\Vert A \Vert_{\infty} \coloneqq \max_{i \in [m]} \sum\nolimits_{j=1}^n |A_{ij}|, \qquad \kappa_{\infty}(A) \coloneqq \Vert A \Vert_{\infty} \, \Vert A^{\dag} \Vert_{\infty},
\]
then 
Bauer \cite{Bauer-63, Bauer-69} and van der Sluis \cite{vdSluis-69} proved that we can maximally reduce  $\kappa_{\infty}(A)$ by the row balancing strategy; the optimal diagonal preconditioner is 
\[
X = \diag\Big(\frac{1}{r_1}, \frac{1}{r_2}, \dots, \frac{1}{r_m}\Big),
\]
where $r_i \coloneqq \sum_{j=1}^n |A_{ij}|$ is 
the $\ell_1$ norm of the $i$-th row of $A$,  $i \in [m]$. 
Moreover, there exists a closed-form expression for the optimal condition number. 
If $T_m(\CC)$ is the group of $m \times m$ invertible diagonal matrices, then 
\begin{equation}
\label{eq:Skeel}
	\inf_{X \in T_m(\CC)} \kappa_{\infty}(XA) = \kappa_S(A) \coloneqq \Vert\, |A^{\dag}| \, |A| \, \Vert_{\infty},
\end{equation}
where $|A|$ is a matrix obtained from $A$ if we replace every element with its absolute value. 
The quantity $\kappa_S(A)$ is known as the \emph{Skeel condition number} \cite{Skeel-79}.
Skeel argued that $\kappa_S(A)$ is sometimes a more meaningful measure of a matrix's ill-conditioning. 
For example, if $A$ is an invertible diagonal matrix, then the system $Ax=b$ can be solved with perfect accuracy and $\kappa_S(A)$ always equals $1$; however, the Euclidean condition number $\kappa(A)$ can be arbitrarily large.  This suggests that $\kappa_S(A)$, sometimes, 
captures better the complexity of solving linear systems,
e.g., for triangular matrices.
We also note that $\kappa_S(A)$ is invariant under left scaling by a diagonal matrix, since
\[
\kappa_S(XA) = \Vert \, |A^{\dag}X^{-1}| \, |XA| \, \Vert_{\infty} = \Vert \, |A^{\dag}| \, |X^{-1}| \, |X| \, |A| \, \Vert_{\infty} = \kappa_S(A),
\]
where the last equality follows from $|X^{-1}| \, |X| = |X|^{-1} |X| = I_m$, since $X$ is diagonal.
The invariance of $\kappa_S$ under diagonal scaling also follows from \eqref{eq:Skeel}.

On the other hand, if we consider the problem of 
optimal diagonal preconditioning using the Euclidean condition number,  \cref{eq:condition-number-linear}, 
then we do not have a closed-form solution, like in \cref{eq:Skeel}.
Fortunately, we can formulate the problem as
a convex optimization problem \cite{Greenbaum-Rodrigue-1989, Shapiro-82, Watson-91, Braatz-Morari-94}.
To do so, we proceed as follows: First, we notice that it suffices to consider diagonal preconditioners with real and positive entries, since $\kappa(UAV)=\kappa(A)$ for unitary matrices $U$ and $V$.
Then, we introduce the change of variables $X=e^{H_1}$ and $Y=e^{H_2}$; now the problem reduces to optimize $\kappa(e^{H_1}\, A\, e^{-H_2})$ over real diagonal matrices $H_1$ and $H_2$. 
Following \cite{Braatz-Morari-94}, for any invertible matrix $A$, it holds
\begin{equation}
\label{eq:diagonal-convex-program}
\inf_{ H_1, H_2\in\mathrm{Diag}_n(\RR)}\; \kappa(e^{H_1}\, A\, e^{-H_2}) = \inf_{H\in \mathrm{Diag}_{2n}(\RR)} \; \Big\Vert e^{H} \, \begin{pmatrix}
	0 & A \\
	A^{-1} & 0
\end{pmatrix}\, e^{-H} \Big\Vert^2. 
\end{equation}
Thus, we reformulate the condition number minimization over diagonal matrices to a norm minimization problem.
Sezginer and Overton \cite{Sezginer-Overton-90} proved that, for any fixed matrix $M$, the operator norm $\Vert e^{H} M e^{-H}\Vert$ is a convex function in the entries of the diagonal matrix $H$.
Hence, \cref{eq:diagonal-convex-program} implies that the optimal diagonal preconditioner is a solution of a convex program.

Building on these convexity results, recent works \cite{QGHYZ-22, JLMSST-23, GQUY-24} find the optimal diagonal preconditioner
by semidefinite programming. 
Indeed, we can compute the optimal condition number under diagonal scaling, $\kappa^\star\coloneqq\inf_{X,Y \text{ diagonal}}\kappa(XAY^{-1})$, by solving the program \[
\begin{split}
    \min_{\kappa>0, X,Y\succ 0} \quad &\kappa\\
    \text{s.t.}\qquad  & X \preceq A^T Y A \preceq \kappa X\\
    & X,Y \text{ diagonal} \enspace.
\end{split}
\] 
If we write $X=\sum_{i=1}^n x_i E_{ii}$ and $A^T Y A = \sum_{i=1}^n y_i a_i a_i^T$, where $a_i$'s are the rows of $A$, 
then the inequalities $X\preceq A^T YA\preceq\kappa_0 X$ become linear matrix inequalities for fixed $\kappa_0$ 
 and the minimization problem becomes an SDP with optimal value $(\kappa^\star)^2$.

Qu et al~\cite{QGHYZ-22} suggest an algorithm that
exploits binary search (by varying $\kappa_0$)
and techniques from interior point algorithms. 
The latter includes tracing numerically
a curve that lies in the interior of the feasible region.
We do so using a predictor corrector algorithm
that in turn is a variant of (a) Newton operator. 
In general, the algorithm, to approximate a matrix, that is within~$\epsilon$ from the matrix with the optimal condition number,
performs $O(\lg\tfrac{1}{\epsilon})$ iterations
and each iteration requires multiple matrix inversions \cite[Theorem~1]{QGHYZ-22}.
If we consider only positive definite matrices $A$, 
then Qu et al~\cite{QGHYZ-22} showed that the optimal condition number
is the solution of the following optimization problem 
\[
\kappa^{*}(A) := \min_{D \succ 0} \kappa( \sqrt{D} \, A \, \sqrt{D}),
\]
that in turn corresponds to a dual SDP.
In this way, they introduce a very efficient practical variant, 
able to precondition very large matrices \cite{GQUY-24}.

We can solve semidefinite programs using a generic SDP solver;
unfortunately its computational complexity is prohibitive
for the problems we consider.
Jambulapati et al~\cite{JLMSST-23},
exploited even further the structure of the problem
and used a special-purpose SDP solver to achieve
a $(\kappa^*)^{1.5}$ dependency on the complexity bounds. 
The corresponding SDP is a variant of mixed packing-covering  SDP
where the packing and covering matrices, and the constraints
are multiples of each other.
Similarly, the results in \cite{Marechal-Ye-09, Lu-Pong-11, CWY-11} show that finding the matrix with the smallest condition number is in a compact, convex subset of the cone of positive definite matrices also admits a SDP formulation.

Despite the recent important progress on optimal diagonal preconditioning
of positive definite matrices, we are still lacking similar results
for non-diagonal preconditioning, say for example block diagonal
(with exception of \cite{demmel-block-precond-23} for positive definite matrices),
let alone for general, say full rank, matrices. 
The reason seems to be that the problems are not convex or quasi-convex anymore. 
Thus, we need to view the problem with different lenses to advance further.

\subsection*{Preconditioning polynomial systems}
Solving systems of polynomial equations is a fundamental problem with theoretical and practical implications in the whole range of science and engineering.
Now, we consider the preconditioning of polynomial systems; a very important problem for which related results are very scarce.
There are numerical (iterative) algorithms to approximate the solutions of a polynomial system, based on homotopy or subdivision, e.g.,~\cite{BC-condition-bk,AME-11,subdivison-mourrain,subdivision-yap,malajovich2023complexity} . 
The complexity and efficiency of both are based on condition numbers, similarly to the linear system case. Hence, preconditioning is a fundamental concept for solving  nonlinear systems as well. 
The important difference with the linear case is that the complexity of solving polynomial systems is rather high, hence we can afford more costly (compared to the linear case) techniques for preconditioning. 

Consider the polynomial system $\f=(f_1,\dots,f_m)$,
where $f_i\in \CC[x_1,\dots,x_n]$, for $i \in [m]$, is a polynomial in $n$ variables with complex coefficients.
Let $\xi\in\CC^n$ be a solution of $\f$ of multiplicity one\footnote{This means that the Jacobian of $\f$ at $\xi$ has rank $\min(m,n)$.}.
The \emph{local condition number} of $(\f,\xi)$ is \begin{equation}
    \mu(\f,\xi) \coloneqq \Vert \f\Vert_W \, \Vert \, D^{\dag}_{\xi}(\f)\,\Vert,
\end{equation} 
where $\Vert\f\Vert^2_W=\sum_{i=1}^n \Vert f_i\Vert^2_W$ is the Bombieri-Weyl norm of $\f$, see \cite[Chapter~16]{BC-condition-bk} and
\cite[Def.~1]{EPR-new-pce-19}, and $D_{\xi}(\f)$ denotes the Jacobian of $\f$ at $\xi$. 

As in the case of linear systems, the preconditioner involves a pair of invertible matrices,  $(X,Y)$.
The first matrix, $X$, \emph{shuffles} the polynomial equations, 
while the second \emph{suffles}, that is, performs a linear change of the variables. 
Hence, $X\cdot \f$ is the transformed system 
\begin{equation}
\label{eq:non-lin-X-act}
X\cdot\f =\, \Big(\; \sum\nolimits_{i=1}^m X_{1i} f_i\, ,\,\sum\nolimits_{i=1}^m X_{2i} f_i ,\, \dots, \,\sum\nolimits_{i=1}^m X_{mi} f_i  \,\Big),
\end{equation}
and the action of $Y$ on the original system $\f$ and $\xi$ is as follows: 
\begin{equation}
\label{eq:non-lin-Y-act}
	\f \mapsto (f_1\circ Y^{-1}, f_2\circ Y^{-1},\dots,f_m\circ Y^{-1}), \qquad \xi \mapsto Y\, \xi.
\end{equation}
Notice that if $\xi$ is a simple root of $\f$ and $(X,Y)\cdot (\f,\xi)=(\g, \zeta)$, then $\zeta$ is a simple root of $\g$.
This follows from applying the chain rule on the Jacobian $D_{\xi}(\f)$: 
\begin{equation}
\label{eq:Jacobian-transformation}
    D_{Y\xi}\big( (X,Y)\cdot \f \big) = X D_{\xi}(\f) Y^{-1}.
\end{equation}
The preconditioning aims to provide matrices $X$ and $Y$, such that $\mu((X,Y)\cdot \f, Y\xi)<\mu(\f, \xi)$.

A common preconditioning technique, unfortunately heuristic and without theoretical guarantees, for nonlinear systems is 
the \textbf{Inverse Jacobian scaling (IJS):}
for square systems, i.e., when $m=n$, this technique sets $X=D_{\xi}^{-1}\f$, where $D_{\xi}\f$ is the Jacobian of $\f$ at $\xi$.
This preconditioner satisfies $D_{\xi}(X\cdot \f)=I_m$,
where $I_m$ is the $m \times m$ identity matrix, and it is used to improve the numerical stability of subdivision-based methods \cite{AME-11,Mourrain-Pavone-09, subdivision-yap}.
Another heuristic technique is coefficient balancing (analogous to matrix equilibration) that is sometimes used to reduce ill-conditioning \cite{gtv-coef-balan-00,morgan2009solving}.

\begin{example}
 Consider the following square (as many polynomials as variables) polynomial system 
 $\f= \{f_1 \coloneqq x^2+x+y,  f_2 \coloneqq y^2+x-y \} $.
 At  $\xi=(0,0)$, we have $D_\xi(\ff)=\begin{pmatrix}
        1 & 1\\
        1 & -1
    \end{pmatrix}$
    and its inverse is $D_\xi(\ff)^{-1}=\frac{1}{2}D_{\xi}(\ff)$.
The IJS of $\ff$ is 
\[
    \tilde{\ff} = \begin{cases}
        \tilde{f}_1 \coloneqq \frac{f_1+f_2}{2}= \frac{1}{2}x^2+\frac{1}{2}y^2+x\\
        \tilde{f}_2 \coloneqq \frac{f_1-f_2}{2}= \frac{1}{2}x^2-\frac{1}{2}y^2+y.
    \end{cases}
    \] 
    
 It holds that $\mu(\tilde{\ff},\xi)=2$ but $\mu(\ff,\xi)=\sqrt{3}<2$.
As in the linear systems, this behaviour is quite general: 
   If we sample random systems $\f$ with fixed degree which vanish on $\xi=0$ by setting the constant coefficient of each $f_i$ to be zero and 
   sample the other coefficients from the normal distribution, the experiments show that
   IJS increases the local condition number. 
 \end{example}

To the best of our knowledge, there is no theoretically justified technique
for preconditioning polynomial systems, let alone for optimal preconditioning.
We take on this challenge and aim to fill this gap. Even more, we develop preconditioning techniques that respect sparsity patterns in the system of equations; quite often polynomial systems that appear in applications have a specific sparse structure (not all monomials appear in the input polynomials) rather than being a generic degree $d$ equation in $n$ variables
that involve all the possible monomials.

Here is the roadmap for the paper: \Cref{vocab} introduces necessary vocabulary and states our results precisely.   \Cref{sec:additional-experiments} presents some experimental results with further results included in \Cref{sec:detailed-experiments}. \Cref{hardcore} builds the technical machinery for optimizing over Lie groups.  \Cref{sec:linear-systems} proves the results related to linear preconditioning. \Cref{sec:polysystems} establishes our results related to preconditoning (non-linear) polynomial systems.

%% file: full-experiments.tex
\section{A more detailed experimental section}
\label{sec:detailed-experiments}

We present experimental results to 
evaluate the performance of our preconditioning algorithms.
Our implementation
relies  on the Julia packages \texttt{Manopt.jl} and \texttt{Manifolds.jl} \cite{Manoptjl,Manifoldsjl}, which enable efficient geodesically convex optimization over matrix manifolds.

We evaluate the effect of preconditioning on the Frobenius condition number $\kappa_F$. We generate a dataset of 1000 random matrices of size $500 \times 500$, with entries drawn independently from the standard normal distribution. For each matrix, we compute both the optimal diagonal preconditioner and the optimal block-diagonal preconditioner, using a fixed block size of $5 \times 5$.
The results appear in \cref{fig:Gauss-Frobenius-BlockvsDiag}
and demonstrate consistent improvement of the condition number for both schemes. On average, block-diagonal preconditioning reduces the Frobenius condition number by a factor of approximately $1.6$. In comparison, diagonal preconditioning yields a more modest reduction factor around $1.16$. This indicates that block-diagonal preconditioning improves the conditioning by roughly $1.38$ times more than its diagonal counterpart, 
and is an experimental indication of the benefits of block structure.

\begin{figure}[htbp]
  \centering
  \begin{subfigure}[b]{0.32\textwidth}
    \centering
    \includegraphics[width=\textwidth]{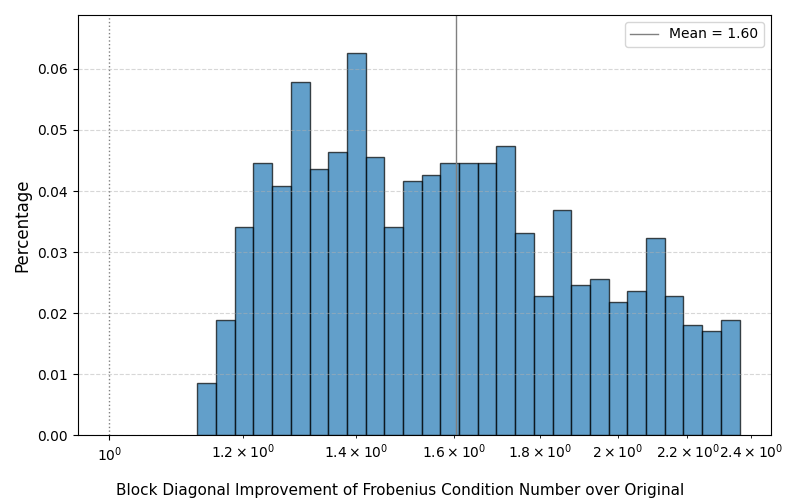}
    \phantomsubcaption
  \end{subfigure}
  \hfill
  \begin{subfigure}[b]{0.32\textwidth}
    \centering
    \includegraphics[width=\textwidth]{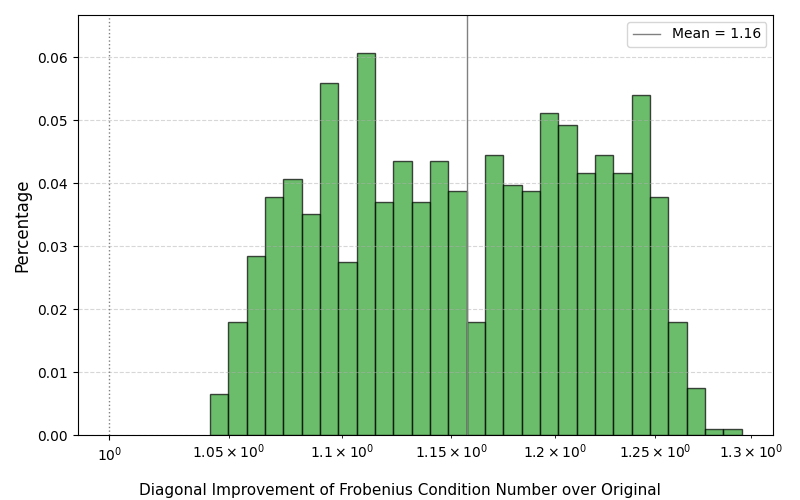}
    \phantomsubcaption
  \end{subfigure}
  \hfill
  \begin{subfigure}[b]{0.32\textwidth}
    \centering
    \includegraphics[width=\textwidth]{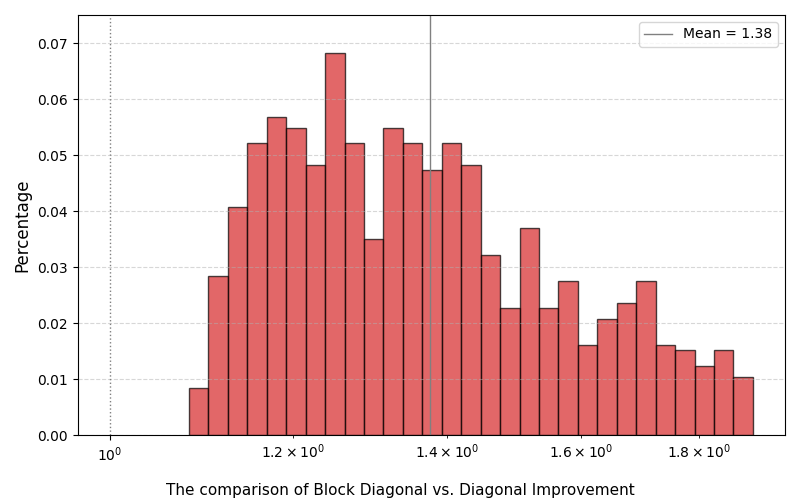}
    \phantomsubcaption
  \end{subfigure}
  \caption{The figure illustrates the improvement in the Frobenius condition number, defined as $\kappa_F(A)/\kappa_F^\star$, for block-diagonal preconditioning with $5 \times 5$ blocks (blue) and diagonal preconditioning (green), applied to $500 \times 500$ matrices with i.i.d.\ standard normal entries. On average, the block-diagonal preconditioner achieves a reduction factor of $1.6$, compared to $1.16$ for the diagonal case. The red curve shows the ratio of these improvements, quantifying the relative advantage of block-diagonal preconditioning.}
  \label{fig:Gauss-Frobenius-BlockvsDiag}
\end{figure}

We further examine the impact of the optimal Frobenius preconditioner on the Euclidean condition number $\kappa(A)$. 
As shown in \cref{fig:correlation-Gauss}, the block-diagonal preconditioner optimized for $\kappa_F$ also significantly reduces $\kappa$, with a moderately strong empirical correlation ($\rho>0.68$) between the two improvements. 
Furthermore, the optimal Frobenius preconditioner reduces the Euclidean condition number in all cases, and improves it by $1.29$ on average.
This suggests that the Frobenius-optimal preconditioner can serve as an effective preconditioner for $\kappa$.

\begin{figure}[htbp]
  \centering
  \begin{subfigure}[b]{0.45\textwidth}
    \centering
    \includegraphics[width=\textwidth]{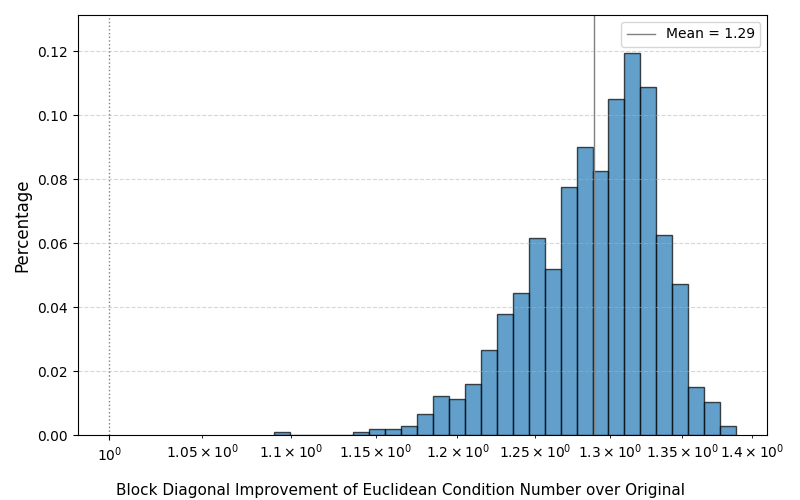}
    \phantomsubcaption
  \end{subfigure}
  \hfill
  \begin{subfigure}[b]{0.32\textwidth}
   \centering   
  \includegraphics[width=\textwidth]{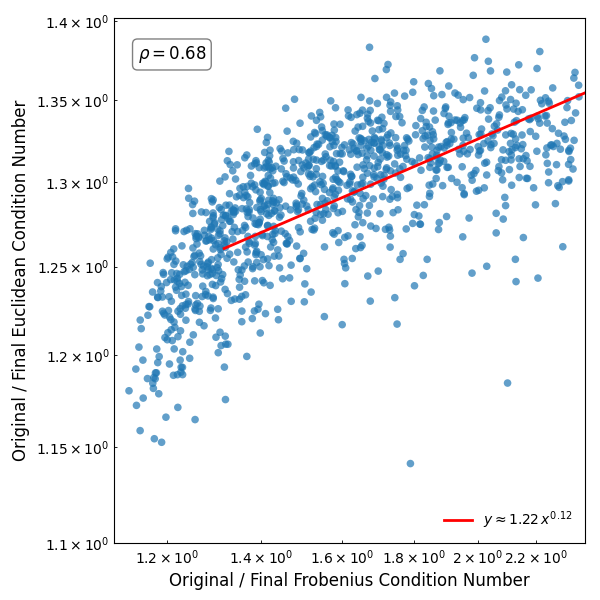}
  \phantomsubcaption
  \end{subfigure}
  \caption{The correlation between the improvement on the Frobenius and the Euclidean condition number.}
  \label{fig:correlation-Gauss}
\end{figure}

In our second experiment, we evaluate the preconditioners on $425$ sparse matrices from the \texttt{SuiteSparse} collection, each with $\leq 1000$ rows. Consistent with the Gaussian case, block-diagonal preconditioning yields substantial improvements:
\begin{itemize}
\item The Frobenius condition number $\kappa_F$ improves by $\approx 1602\times$ on average with block-diagonal preconditioning, compared to $\approx 600\times$ for diagonal preconditioning.
\item The block-diagonal preconditioner's improvement is $2.43\times$ greater than that of the diagonal preconditioner, further demonstrating its superiority.
\end{itemize}

\begin{figure}[htbp]
  \centering
  \begin{subfigure}[b]{0.32\textwidth}
    \centering
    \includegraphics[width=\textwidth]{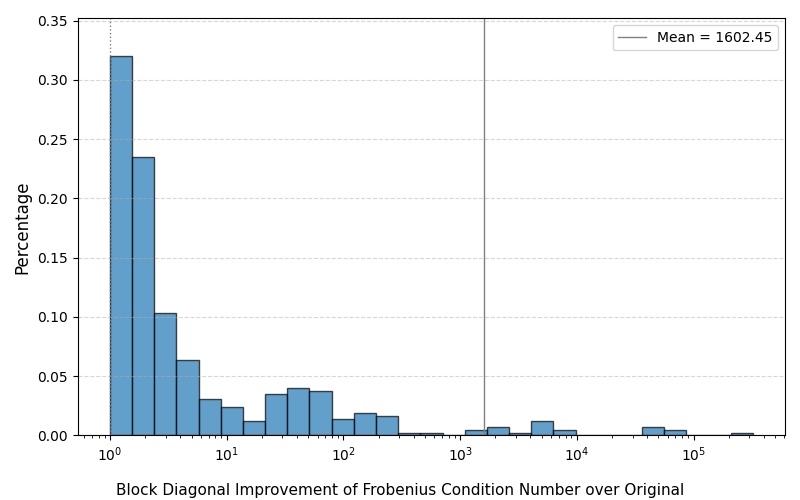}
    \phantomsubcaption
  \end{subfigure}
  \hfill
  \begin{subfigure}[b]{0.32\textwidth}
    \centering
    \includegraphics[width=\textwidth]{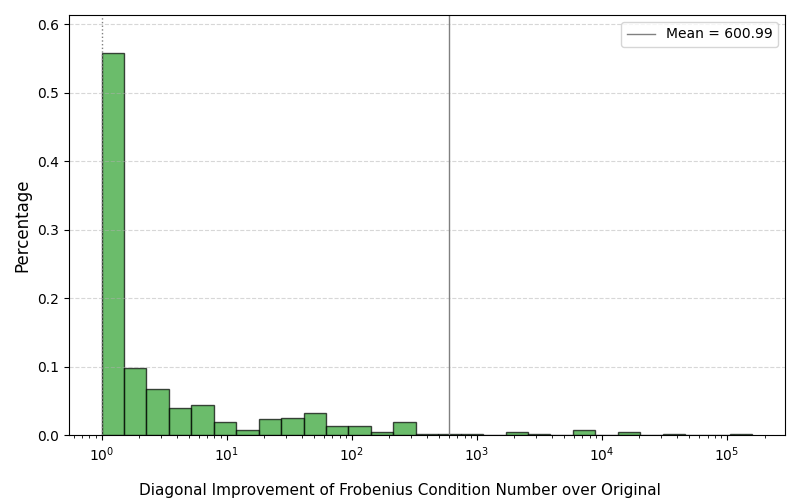}
    \phantomsubcaption
  \end{subfigure}
  \hfill
  \begin{subfigure}[b]{0.32\textwidth}
    \centering
    \includegraphics[width=\textwidth]{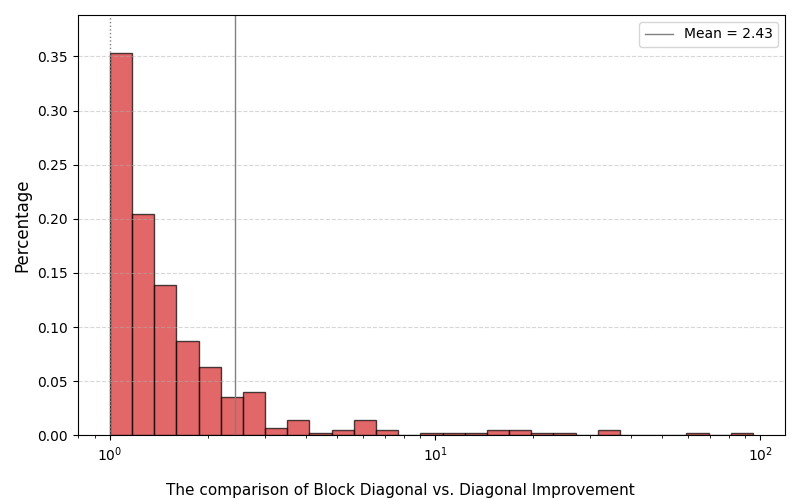}
    \phantomsubcaption
  \end{subfigure}
  \caption{The same as \cref{fig:Gauss-Frobenius-BlockvsDiag}, for the \texttt{SuiteSparse} dataset. }
  \label{fig:Sparse-Frobenius-BlockvsDiag}
\end{figure}

We further investigate whether optimal Frobenius preconditioners-designed to minimize $\kappa_F$-also improve the Euclidean condition number $\kappa$. 
For the \texttt{SuiteSparse} matrices, we observe a strong correlation ($\rho \approx 0.96$) between $\kappa_F$ and $\kappa$ improvements, with optimal Frobenius preconditioner reducing $\kappa$ in $\approx 95\%$ of the data.
\cref{fig:correlation-Sparse} shows this relationship. 
Combined with the Gaussian results \cref{fig:correlation-Gauss} , this confirms that Frobenius preconditioner consistently reduces the Euclidean condition number, $\kappa$.

\begin{figure}[htbp]
  \centering
  \begin{subfigure}[b]{0.45\textwidth}
    \centering
    \includegraphics[width=\textwidth]{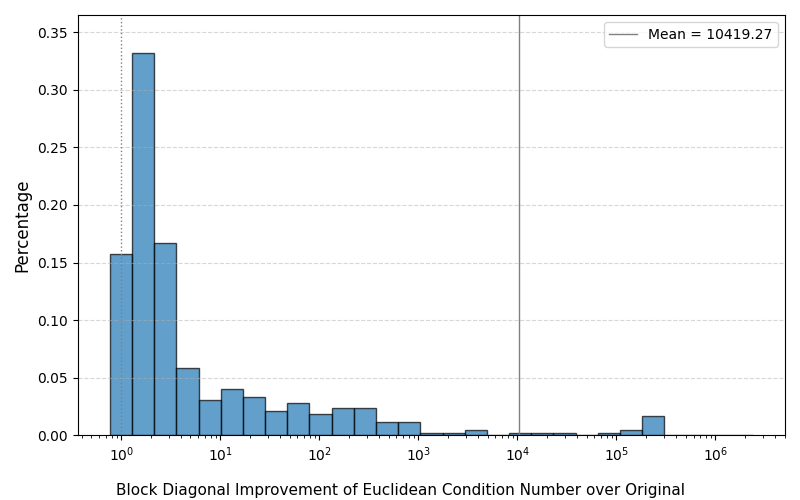}
    \phantomsubcaption
  \end{subfigure}
  \hfill
  \begin{subfigure}[b]{0.32\textwidth}
   \centering   
  \includegraphics[width=\textwidth]{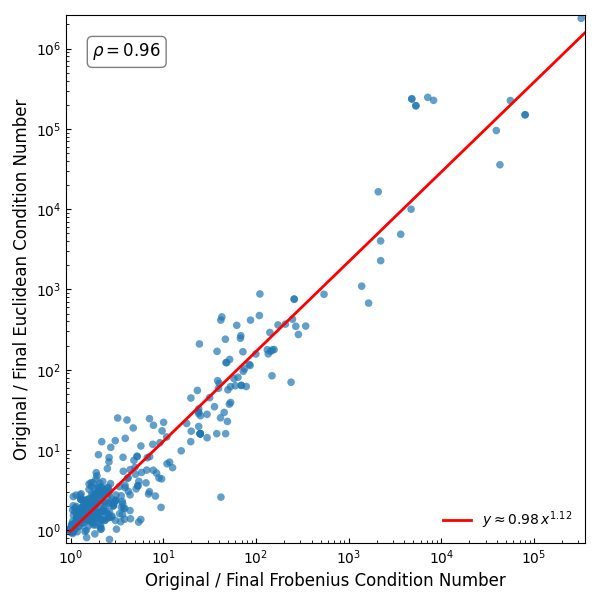}
  \phantomsubcaption
  \end{subfigure}
  \caption{The correlation between the improvement on the Frobenius and the Euclidean condition number for the \texttt{SuiteSparse} dataset.}
  \label{fig:correlation-Sparse}
\end{figure}